\documentclass[a4paper,11pt,reqno,twoside]{amsart}
\usepackage{graphicx}
\usepackage{epsbox}
\usepackage{amscd}

\usepackage{color}
\usepackage{ifpdf}
\usepackage{array}
\usepackage{slashbox}
\usepackage[T1]{fontenc}
\usepackage[latin1]{inputenc}
\usepackage{mathtools}
\usepackage{amsthm,amssymb}
\usepackage{url}

\newcommand*{\C}{\mathbb{C}}
\newcommand*{\R}{\mathbb{R}}
\newcommand*{\Q}{\mathbb{Q}}
\newcommand*{\Z}{\mathbb{Z}}

\renewcommand*{\geq}{\geqslant}
\renewcommand*{\leq}{\leqslant}
\newtheoremstyle{erdfn}
  {}
  {}
  {\itshape}
  {}
  {\bfseries}
  {}
  { }
  {}
\newtheoremstyle{erthm}
  {}
  {}
  {\itshape}
  {}
  {\bfseries}
  {}
  { }
  {}
\newtheoremstyle{errem}
  {}
  {}
  {}
  {}
  {\bfseries}
  {}
  { }
  {}
\theoremstyle{erthm}
\newtheorem{theorem}{Theorem}[section] 
\newtheorem{corollary}[theorem]{Corollary}
\newtheorem{proposition}[theorem]{Proposition}
\newtheorem{lemma}[theorem]{Lemma}

\theoremstyle{erdfn}
\newtheorem{definition}[theorem]{Definition}

\theoremstyle{errem}

\newtheorem{remark}{Remark}

\numberwithin{equation}{subsection}
\title[On zeros of self-reciprocal polynomials]%
      {On zeros of self-reciprocal polynomials} 
\author[M. Suzuki]{Masatoshi Suzuki}
\date{Version of \today}
\subjclass[2000]{30C15, 34A55, 34L40}
\keywords{}
\AtBeginDocument{%
\mathtoolsset{showonlyrefs,mathic = true}
\begin{abstract}
We establish a necessary and sufficient condition 
for which all zeros of a self-reciprocal polynomial lie on the unit circle. 
Moreover, we relate the necessary and sufficient condition 
with a canonical system of linear differential equations (in the sense of de Branges). 
This relationship enable us to understand that the property of a self-reciprocal polynomial 
having only zeros on the unit circle 
is equivalent to the positive semidefiniteness of Hamiltonians 
of corresponding canonical systems.  
\end{abstract}
\maketitle
}
\begin{document}
%
%
\section{Introduction}
%
%
A nonzero polynomial $P(x)=c_0 x^n + c_1 x^{n-1} + \cdots + c_{n-1} x + c_n $ 
with {\it real coefficients} 
is called a {\it self-reciprocal polynomial} of degree $n$ 
if $c_0\not=0$ and $P(x)=x^n P(1/x)$,
or equivalently, $c_0\not=0$ and $c_{k}=c_{n-k}$ for every $0\leq k \leq n$. 
The zeros of a self-reciprocal polynomial either lie on the unit circle, $T=\{z \in \C\,|\, |z|=1\}$,  
or are distributed symmetrically with respect to $T$. 
Therefore, one of the basic problem 
is to find a ``nice'' condition of coefficients of a self-reciprocal polynomial for which all its zeros lie on $T$. 
In this paper, we study this problem for 
self-reciprocal polynomials of {\it even degree}. 
The restriction on the degree is not essential, because 
if $P(x)$ is self-reciprocal and of odd degree, 
then there exists a self-reciprocal polynomial $\tilde{P}(x)$ of even degree 
and an integer $r \geq 1$ 
such that $P(x)=(x+1)^{r}\tilde{P}(x)$. 
In contrast, the realness of coefficients are crucial to the results and methods of this paper.  
\medskip

A nonzero polynomial $P(x)=c_0 x^n + c_1 x^{n-1} + \cdots + c_{n-1} x + c_n$ $(c_i \in \C)$ 
is called a {\it self-inversive polynomial} of degree $n$ 
if $c_0\not=0$ and $P(x)=\mu x^n\overline{P(1/\bar{x})}$ for some constant $\mu \in \C$ of modulus $1$,
or equivalently, $c_0\not=0$ and $c_{k}=\mu \overline{c_{n-k}}$ for every $0\leq k \leq n$. 
If all coefficients of a self-inversive polynomial are real, then $\mu$ must be $\pm 1$. 
Therefore, self-reciprocal polynomials are special cases of self-inversive polynomials.   

The Gauss--Lucas theorem~\cite{Lucas74} assert that the convex hull of the zeros of 
any nonconstant complex polynomial contains the zeros of its derivative. 
Further, by a result of Schur~\cite{Schur17} and Cohn~\cite{Cohn22}, 
all zeros of a self-inversive polynomial $P(x)$ lie on $T$ 
if and only if all zeros of $P'(x)$ lie inside or on $T$. 
Chen~\cite{Chen95} found another necessary and sufficient condition 
for which all zeros of a self-inversive polynomial lie on $T$. 
It was rediscovered by Lal{\'i}n--Smyth \cite{LalinSmyth12}
(See the book of Marden~\cite{Marden66} and the survey paper of Milovanovi\'c--Rassias~\cite{MiloRass00}
for several systematic treatments of zeros of polynomials.)
The above results are often useful in order to examine whether all zeros of a given self-inversive (or self-reciprocal) polynomial lie on $T$, 
but a sufficient condition in terms of its coefficient is sometimes more convenient, 
because it may not be easy to check the above criterions for a given polynomial. 

Lakatos~\cite{Lakatos02} presented a simple sufficient condition 
of coefficients of a self-reciprocal polynomial 
for which all its zeros lie on $T$. 
She proved that if $P(x) \in \R[x]$ is a self-reciprocal polynomial of degree $n \geq 2$ 
satisfying  
\[
|c_0| \geq \sum_{k=1}^{n-1} |c_k - c_0|,
\]
then all zeros of $P(x)$ lie on $T$ 
(see Lakatos--Losonczi~\cite{LakatosLosonczi09}, Kwon~\cite{Kwon11} 
and their references for further generalization 
of this kind of conditions). 
Another simple sufficient condition in terms of coefficients 
was found by Chen~\cite{Chen95} and Chinen~\cite{Chinen08}, independently. 
They proved that if $P(x) \in \R[x]$ has the form
\[
P(x) = 
(c_0 x^n + c_1 x^{n-1} + \cdots + c_{k} x^{n-k}) 
+ (c_k x^k + c_{k-1} x^{k-1} + \cdots + c_1 x+  c_0),
\]
and $c_0 > c_1 > \cdots > c_k >0$ ($n \geq k$), 
then all zeros of $P(x)$ lie on $T$. 
As well as these results, known conditions of coefficients of general self-reciprocal polynomials 
for which all its zeros lie on $T$  
are usually sufficient conditions and they are not applied to the case that 
coefficients of middle terms are quite large 
comparing with coefficients of both ends. 

\medskip

The first result of this paper is the construction of a set of $2g$ rational functions 
$R_n(\underline{c})=R_{g,n}(\underline{c})\in \Q(\underline{c})$ ($1 \leq n \leq 2g$) 
of $(g+1)$ variables 
$\underline{c}=(c_0,\cdots,c_g)$ 
for every positive integer $g$ 
such that all zeros of a self-reciprocal polynomial  
\begin{equation} \label{def_Pg}
P_g(x) 
= \sum_{k=0}^{g-1} c_k(x^{(2g-k)} + x^{k}) + c_g x^{g}
\end{equation}
with real coefficients $\underline{c}=(c_0,\cdots,c_g) \in \R^{g+1}$ 
($c_0 \not=0$)
lie on $T$ and {\it simple} if and only if 
$R_{n}(\underline{c})$ is a finite positive real number for every $1 \leq n \leq 2g$ (Theorem \ref{thm_2}).   
The rational functions $R_{n}(\underline{c})$ are constructed in Section \ref{section_2_2}
by using a linear system introduced in Section \ref{section_2_1}. 
The second result is a variant of the first result. 
We construct a set of $2g$ rational functions 
$R_n(\underline{c}\,;q^\omega)=R_{g,n}(\underline{c}\,;q^\omega)\in \Q(\underline{c},q^\omega)$ ($1 \leq n \leq 2g$) 
of $(g+2)$ variables 
$\underline{c}=(c_0,\cdots,c_g)$ and $q^\omega$ 
for every positive integer $g$ 
such that all zeros of a self-reciprocal polynomial \eqref{def_Pg} 
with real coefficients $\underline{c}=(c_0,\cdots,c_g) \in \R^{g+1}$ 
($c_0 \not=0$)
lie on $T$ if and only if 
$R_{n}(\underline{c}\,;q^\omega)$ is a finite real number  
for every $1 \leq n \leq 2g$ and $\omega>0$, 
where $q$ is an arbitrary fixed real number larger than one 
(Theorem \ref{thm_4}). 
Rational functions $R_n(\underline{c}\,;q^\omega)$ are constructed 
in Section \ref{section_2_3} by a way similar to $R_n(\underline{c})$. 
Moreover, it is proved that 
\[
\lim_{q^\omega \to 1^+} R_n(\underline{c}\,;q^\omega) =R_n(\underline{c})
\]
as a rational function of $\underline{c}=(c_0,\cdots,c_g)$ (Theorem \ref{thm_5}). 

Subsequently, we attempt to understand the above positivity conditions 
from a viewpoint of the theory of canonical systems of ordinary linear differential equations. 
After the review of the theory of canonical systems in Section \ref{section_3}, 
we construct two kinds of systems of ordinary linear differential equations attached to self-reciprocal polynomials 
in Section \ref{section_4} (Theorem \ref{thm_6} and \ref{thm_7}). 
It is proved that the system of Section \ref{section_4_1} (resp. Section \ref{section_4_2}) is a canonical systems 
if and only if 
all zeros of a corresponding self-reciprocal polynomial \eqref{def_Pg} 
with real coefficients lie on $T$ and simple (resp. lie on $T$) (Corollary \ref{cor_1} (resp. Corollary \ref{cor_2})). 

We prove the result of Section \ref{section_4_1} (Theorem \ref{thm_6}) in Section \ref{section_5} 
after a preparation of two basic tools. 
By using Theorem \ref{thm_6} and tools in Section \ref{section_5}, 
we prove the results of Section \ref{section_2_2} in Section \ref{section_6}. 
Results of Section \ref{section_4_2} are proved in Section \ref{section_7} 
by a way similar to the proofs in Section \ref{section_5} and \ref{section_6}. 
Finally, we comments on important remaining problems in Section \ref{section_8}.  
\smallskip

This paper is written in self-contained fashion as much as possible. 
\medskip

\noindent
{\bf Acknowlegements}~
This work was supported by KAKENHI (Grant-in-Aid for Young Scientists (B)) No. 21740004. 
%
%
\section{Results I: \\ 
Two kinds of functions 
attached to self-reciprocal polynomials} \label{section_2}
%
%
In this section, 
we construct two kinds of rational functions $R_{n}(\underline{c})$ and $R_{n}(\underline{c}\,;q^\omega)$ 
by using a linear system in order to state a necessary and sufficient conditions 
for all zeros of a self-reciprocal polynomial to lie on $T$. 
Moreover, we clarify the relation between $R_{n}(\underline{c})$ and $R_{n}(\underline{c}\,;q^\omega)$. 
Throughout this section, we fix a positive integer $g$. 

%
%
\subsection{Linear System} \label{section_2_1}
%
%
Firstly, we define the $(2k+2)\times (2k+2)$ matrix $P_k(m_k)$ endowed with the parameter $m_k$
and the $(2k+2)\times(2k+4)$ matrix $Q_k$  for every nonnegative integer $k$ as follows. 
For $k=0,1$, we define 
\[
P_0 := 
\left[
\begin{array}{c|c}
1 & 0 \\ \hline 0 & 1 
\end{array}\right],
\quad 
Q_0 := 
\left[
\begin{array}{cc|cc}
1 & 1 & 0 & 0 \\ \hline 0 & 0 & 1 & -1
\end{array}\right],
\]
\[
P_1(m_1) := \left[
\begin{array}{cc|cc}
1 & 0 & 0 & 0 \\ 
0 & 1 & 0 & 0 \\ \hline
0 & 0 & 1 & 0 \\ \hline
0 & 1 & 0 & -m_1 
\end{array}\right], 
\quad 
Q_1 := \left[
\begin{array}{ccc|ccc}
1 & 0 & 1 & 0 & 0 & 0 \\ 
0 & 1 & 0 & 0 & 0 & 0 \\ \hline 
0 & 0 & 0 & 1 & 0 & -1 \\ \hline 
0 & 0 & 0 & 0 & 0 & 0 
\end{array}\right].
\]
For $k \geq 2$,  we define $P_k(m_k)$ and $Q_k$ blockwisely as follows 
\[
P_k(m_k) := 
\left[\begin{array}{c|c} 
V_k^+ & \pmb{0} \\[2pt] \hline
\pmb{0} & V_k^- \\[2pt] \hline
\pmb{0}I_k & -m_k\cdot \pmb{0}I_k
\end{array}\right], 
\qquad 
Q_k 
:= 
\left[\begin{array}{c|c} 
W_k^+ & \pmb{0} \\[2pt] \hline
\pmb{0} & W_k^- \\[2pt] \hline 
\pmb{0}_{k,k+2} & \pmb{0}_{k,k+2}
\end{array}\right], 
\]
where 
$\pmb{0}I_k 
:=
\begin{bmatrix} 
\pmb{0}_{k,1} & I_k 
\end{bmatrix}$, 
$-m_k\cdot\pmb{0}I_k 
=
\begin{bmatrix} 
\pmb{0}_{k,1} & -m_k\cdot I_k 
\end{bmatrix}$, 
$\pmb{0}_{k,l}$ is the $k \times l$ zero matrix, 
$I_k$ is the identity matrix of size $k$, 
\[
\aligned
V_k^+ 
& := 
\begin{bmatrix}
1 & 0 & 0 & 0 & \cdots & 0 & 0 & 0 \\ 
0 & 1 & 0 & 0 & \cdots & 0 & 0 & 1 \\
0 & 0 & 1 & 0 & \cdots & 0 & 1 & 0 \\
\cdots & \cdots & \cdots & \cdots & \cdots & \cdots &  \cdots & \cdots \\  
0 & 0 & 0 & 1 & 0 & 1 & 0 & 0 \\
0 & 0 & 0 & 0 & 1 & 0 & 0 & 0
\end{bmatrix} 
\quad \left(\frac{k+3}{2}\right)\times\left(k+1\right) 
\quad \text{if $k$ is odd,}
\\
& := 
\begin{bmatrix}
1 & 0 & 0 & 0 & \cdots & \cdots & 0 & 0 & 0 \\ 
0 & 1 & 0 & 0 & \cdots & \cdots & 0 & 0 & 1 \\
0 & 0 & 1 & 0 & \cdots & \cdots & 0 & 1 & 0 \\
\cdots & \cdots & \cdots & \cdots & \cdots & \cdots & \cdots &  \cdots & \cdots \\  
0 & 0 & 0 & 1 & 0 & 0 & 1 & 0 & 0 \\
0 & 0 & 0 & 0 & 1 & 1 & 0 & 0 & 0
\end{bmatrix}
\quad \left(\frac{k+2}{2}\right)\times\left(k+1\right) 
\quad \text{if $k$ is even,}
\endaligned
\]

\[
\aligned
V_k^- 
& := 
\begin{bmatrix}
1 & 0 & 0 & 0 & \cdots & \cdots & \cdots & 0 & 0 & 0 \\ 
0 & 1 & 0 & 0 & \cdots & \cdots & \cdots & 0 & 0 & -1 \\
0 & 0 & 1 & 0 & \cdots & \cdots & \cdots & 0 & -1 & 0 \\
\cdots & \cdots & \cdots & \cdots & \cdots & \cdots & \cdots & \cdots &  \cdots & \cdots \\  
0 & 0 & 0 & 1 & 0 & 0 & 0 & -1 & 0 & 0 \\
0 & 0 & 0 & 0 & 1 & 0 & -1 & 0 & 0 & 0
\end{bmatrix}
\quad \left(\frac{k+1}{2}\right)\times\left(k+1\right) 
\quad \text{if $k$ is odd}, 
\\
& := 
\begin{bmatrix}
1 & 0 & 0 & 0 & \cdots & \cdots & 0 & 0 & 0 \\ 
0 & 1 & 0 & 0 & \cdots & \cdots & 0 & 0 & -1 \\
0 & 0 & 1 & 0 & \cdots & \cdots & 0 & -1 & 0 \\
\cdots & \cdots & \cdots & \cdots & \cdots & \cdots & \cdots &  \cdots & \cdots \\  
0 & 0 & 0 & 1 & 0 & 0 & -1 & 0 & 0 \\
0 & 0 & 0 & 0 & 1 & -1 & 0 & 0 & 0
\end{bmatrix}
\quad \left(\frac{k+2}{2}\right)\times\left(k+1\right) 
\quad \text{if $k$ is even},
\endaligned
\]
and matrices $W_k^{\pm}$ are defined by adding column vectors ${}^{t}(1~0~\cdots~0)$ 
to the right-side end of matrices $V_k^\pm$: 
\[
\aligned
W_k^+ 
& := 
\begin{bmatrix}
1 & 0 & 0 & \cdots & 0 & 0 & 1 \\ 
0 & 1 & 0 & \cdots & 0 & 1 & 0 \\
\cdots & \cdots & \cdots & \cdots & \cdots &  \cdots & \cdots \\  
0 & 0 & 1 & 0 & 1 & 0 & 0 \\
0 & 0 & 0 & 1 & 0 & 0 & 0
\end{bmatrix} 
\quad \left(\frac{k+3}{2}\right)\times\left(k+2\right) 
\quad \text{if $k$ is odd,}
\\
& := 
\begin{bmatrix}
1 & 0 & 0 & \cdots & \cdots & 0 & 0 & 1 \\ 
0 & 1 & 0 & \cdots & \cdots & 0 & 1 & 0 \\
\cdots & \cdots & \cdots & \cdots & \cdots & \cdots &  \cdots & \cdots \\  
0 & 0 & 1 & 0 & 0 & 1 & 0 & 0 \\
0 & 0 & 0 & 1 & 1 & 0 & 0 & 0
\end{bmatrix}
\quad \left(\frac{k+2}{2}\right)\times\left(k+2\right) 
\quad \text{if $k$ is even,}
\endaligned
\]

\[
\aligned
W_k^- 
& := 
\begin{bmatrix}
1 & 0 & 0 & \cdots & \cdots & \cdots & 0 & 0 & -1 \\ 
0 & 1 & 0 & \cdots & \cdots & \cdots & 0 & -1 & 0 \\
\cdots & \cdots & \cdots & \cdots & \cdots & \cdots & \cdots &  \cdots & \cdots \\  
0 & 0 & 1 & 0 & 0 & 0 & -1 & 0 & 0 \\
0 & 0 & 0 & 1 & 0 & -1 & 0 & 0 & 0
\end{bmatrix}
\quad \left(\frac{k+1}{2}\right)\times\left(k+2\right) 
\quad \text{if $k$ is odd,}
\\
& := 
\begin{bmatrix}
1 & 0 & 0 & \cdots & \cdots & 0 & 0 & -1 \\ 
0 & 1 & 0 & \cdots & \cdots & 0 & -1 & 0 \\
\cdots & \cdots & \cdots & \cdots & \cdots & \cdots &  \cdots & \cdots \\  
0 & 0 & 1 & 0 & 0 & -1 & 0 & 0 \\
0 & 0 & 0 & 1 & -1 & 0 & 0 & 0
\end{bmatrix}
\quad \left(\frac{k+2}{2}\right)\times\left(k+2\right) 
\quad \text{if $k$ is even}.
\endaligned
\]

\begin{lemma} \label{lem_201}
For $k \geq 1$, we have 
\[
\det P_k(m_k)
= 
\begin{cases} 
~\varepsilon_{2j+1}\, 2^{j} m_{2j+1}^{j+1} & \text{if $k=2j+1$}, \\
~\varepsilon_{2j}\, 2^{j} m_{2j}^{j} & \text{if $k=2j$}, 
\end{cases} \\
\]
where 
\[
\varepsilon_{2j+1} 
= 
\begin{cases}
 ~+1 & j \equiv 2, 3 ~{\rm mod}\, 4 \\
 ~-1 & j \equiv 0, 1 ~{\rm mod}\, 4
\end{cases}, \quad 
\varepsilon_{2j} 
= 
\begin{cases}
 ~+1 & j \equiv 0, 1 ~{\rm mod}\, 4 \\
 ~-1 & j \equiv 2, 3 ~{\rm mod}\, 4
\end{cases}.
\]
In particular, $P_k(m_k)$ is invertible if and only if $m_k\not=0$. 
\end{lemma}
\begin{proof}
This is trivial for $k=1$. 
Suppose that $k=2j+1 \geq 3$ and write $P_k(m_k)$ as $(v_1~\cdots~v_{2k+2})$ by its column vectors $v_l$. 
At first, we make $I_{k+2}$ at the left-upper corner by exchanging columns $v_{(k+5)/2},\cdots,v_{k+1}$ 
and $v_{k+2},\cdots, v_{(3k+3)/2}$ so that
\[
\det P_k(m_k) = \det(v_1~\cdots~v_{(k+3)/2}~v_{k+2}~\cdots~v_{(3k+3)/2}~v_{(k+5)/2}~\cdots~v_{k+1}~v_{(3k+5)/2}~\cdots~v_{2k+2}). 
\]
Then, by eliminating every $1$ and $-m_k$ under $I_{k+2}$ of the left-upper corner, we have 
\[
\det P_k(m_k) = 
\det 
\begin{bmatrix}
I_{k+2} & \ast \\
~\pmb{0}_{k,k+2} & Z_k 
\end{bmatrix}. 
\]
Here $Z_k$ is the $k \times k$ matrix 
\[
Z_k = \begin{bmatrix}
Z_{k,1} & Z_{k,2} \\
\pmb{0}_{j+1,j} & Z_{k,3}
\end{bmatrix}, 
\]
for which $Z_{k,1}$ is the $j\times j$ anti-diagonal matrix with $-1$ on the anti-diagonal line and  
\[
\left[\begin{array}{c}
Z_{k,2} \\
\hline
Z_{k,3}
\end{array}\right]
=
\left[\begin{array}{ccccc}
0 & 0 & \cdots & 0 & -m_k \\
\vdots & \vdots & 0 & \ast & 0 \\
0 & 0 & -m_k & 0 & \vdots \\ 
0 & -m_k & 0 & \cdots & 0 \\
\hline
-m_k & 0 & \cdots & 0 & 0 \\
0 & -2m_k & 0 & \ddots & 0 \\
0 & 0 & -2m_k & \ddots & 0 \\
\vdots & \vdots & \ddots & \ddots & 0 \\ 
0 & 0 & \cdots & 0 & -2m_k \\
\end{array}\right] \quad 
\begin{matrix}
j \times (j+1) \\ \\ \\ \\ \\ 
(j+1) \times (j+1).
\end{matrix}
\]
The above formula of $\det P_k(m_k)$ implies the desired result. 
The case of even $k$ is proved by a way similar to the case of odd $k$.     
\end{proof}

Let 
$\underline{z}=(x_0, \cdots, x_{2g}, y_0, \cdots, y_{2g})$ be $(4g+2)$ indeterminate elements. 
We define the column vector $v_g(0)= v_g(\underline{z}\,;0)$ of length $(4g+2)$ by 
\[
v_g(0) 
= {}^{t}\!
\begin{pmatrix}
x_0 & \cdots & x_{2g} & y_0 & \cdots & y_{2g}  
\end{pmatrix},
\]
where ${}^tv$ means the transpose of a row vector $v$, 
and define column vectors $v_g(n)= v_g(\underline{z}\,;n)$ $(1 \leq n \leq 2g)$ of length $(4g-2n+2)$ inductively as follows: 
\begin{equation}\label{def_m1}
m_{2g-n}(\underline{z}):=\frac{v_g(n-1)[1]+v_g(n-1)[2g-n+2]}{v_g(n-1)[2g-n+3]-v_g(n-1)[4g-2n+4]},
\end{equation}
\begin{equation}\label{def_v1}
v_g(n) := P_{2g-n}(m_{2g-n}(\underline{z}))^{-1}Q_{2g-n} \, v_g(n-1),
\end{equation}
where $P_0(m_0):=P_0$ and $v[j]$ means the $j$th component of a column vector $v$. 
By the following lemma, we confirm that $m_{2g-n}(\underline{z})$ and $v_g(n)$ are well-defined 
as a rational function and a vector valued rational function 
of $\underline{z}=(x_0, \cdots, x_{2g}, y_0, \cdots, y_{2g})$ for every $1 \leq n \leq 2g$, 
respectively.  

\begin{lemma} \label{lem_202} 
Let $K=\Q(x_0, \cdots, x_{2g},y_0, \cdots, y_{2g})$. 
For every $1 \leq n \leq 2g$, 
$m_{2g-n}(\underline{z})$ is a nonzero element of $K$ 
and hence $P_{2g-n}(m_{2g-n}(\underline{z})) \in GL_{4g-2n+2}(K)$. 
\end{lemma}
\begin{proof} 
By Lemma \ref{lem_201}, it is sufficient to show that 
there exists a numerical vector $\underline{c} \in \R^{4g+2}$ 
such that $|m_{2g-n}(\underline{c})|<\infty$ and $m_{2g-n}(\underline{c})\not=0$ for every $1 \leq n \leq 2g$. 
Existence of such numerical vector $\underline{c}$ 
is guaranteed by the necessity part of Theorem \ref{thm_1} below, 
since it is clear that there exists a self-reciprocal polynomial of degree $2g$ 
with real coefficients having only simple zeros on $T$ for every positive integer $g$.  
\end{proof}
Here we mention that the vector $v_g(n)$ of \eqref{def_v1} can be defined from $v_g(n-1)$ by a slightly different way 
according to the following lemma. 
\begin{lemma} \label{lem_203}
For every $1 \leq n \leq 2g$, we have
\[
\aligned
\,& 
\frac{v_g(n-1)[1]+v_g(n-1)[2g-n+2]}{v_g(n-1)[2g-n+3]-v_g(n-1)[4g-2n+4]} \\
& \qquad \qquad \qquad \qquad=
\frac{(P_{2g-n}(m_{2g-n})^{-1}Q_{2g-n} \, v_g(n-1))[1]}{(P_{2g-n}(m_{2g-n})^{-1}Q_{2g-n} \, v_g(n-1))[2g-n+2]},
\endaligned
\]
that is, the right-hand side is independent of the indeterminate element $m_{2g-n}$.
\end{lemma}
\begin{proof} Define $(k+1) \times (k+2)$ matrices 
$M_{k,1}$, $M_{k,2}$, $M_{k,3}$, $M_{k,4}$ by
\[
\begin{bmatrix}
M_{0,1} & M_{0,2} \\ M_{0,3} & M_{0,4}  
\end{bmatrix}
=
\begin{bmatrix}
1 & 1 & 0 & 0 \\  0 & 0 & 1 & -1  
\end{bmatrix}, 
\quad 
\begin{bmatrix}
M_{1,1} & M_{1,2} \\ M_{1,3} & M_{1,4}  
\end{bmatrix}
=
\begin{bmatrix}
1 & 0 & 1 & 0 & 0 & 0 \\  
0 & 1 & 0 & 0 & 0 & 0 \\
0 & 0 & 0 & 1 & 0 & -1 \\ 
0 & 1/m_1 & 0 & 0 & 0 & 0
\end{bmatrix} 
\]
and 
\[
\aligned
M_{k,1}
&= \frac{1}{2}
\begin{bmatrix}
2 & 0 & 0 & 0 & \cdots & 0 & 0 & 0 & 2 \\
0 & 1 & 0 & 0 & \cdots  & 0 & 0 & 1 & 0 \\
0 & 0 & 1 & 0 & \cdots  & 0 & 1 & 0 & 0 \\
\cdots & \cdots & \cdots & \cdots & \cdots & \cdots & \cdots & \cdots & \cdots \\
0 & \cdots & 0 & 1 & 0 & 1 & 0 & \cdots & 0 \\
0 & \cdots & \cdots & 0 & 2 & 0 & \cdots & \cdots & 0 \\
0 & \cdots & 0 & 1 & 0 & 1 & 0 & \cdots & 0 \\
\cdots & \cdots & \cdots  & \cdots & \cdots & \cdots & \cdots & \cdots & \cdots \\
0 & 0 & 1 & 0 & \cdots  & 0 & 1 & 0 & 0 \\
0 & 1 & 0 & 0 & \cdots  & 0 & 0 & 1 & 0 
\end{bmatrix} \quad \text{if $k \geq 3$ is odd,} \\
&= \frac{1}{2}
\begin{bmatrix}
2 & 0 & 0 & 0 & \cdots & \cdots & 0 & 0 & 0 & 2 \\
0 & 1 & 0 & 0 & \cdots & \cdots & 0 & 0 & 1 & 0 \\
0 & 0 & 1 & 0 & \cdots & \cdots & 0 & 1 & 0 & 0 \\
\cdots & \cdots & \cdots & \cdots & \cdots & \cdots & \cdots & \cdots & \cdots & \cdots \\
0 & \cdots & 0 & 1 & 0 & 0 & 1 & 0 & \cdots & 0 \\
0 & \cdots & \cdots & 0 & 1 & 1 & 0 & \cdots & \cdots & 0 \\
0 & \cdots & \cdots & 0 & 1 & 1 & 0 & \cdots & \cdots & 0 \\
0 & \cdots & 0 & 1 & 0 & 0 & 1 & 0 & \cdots & 0 \\
\cdots & \cdots & \cdots & \cdots & \cdots & \cdots & \cdots & \cdots & \cdots & \cdots \\
0 & 0 & 1 & 0 & \cdots & \cdots & 0 & 1 & 0 & 0 \\
0 & 1 & 0 & 0 & \cdots & \cdots & 0 & 0 & 1 & 0 
\end{bmatrix} \quad \text{if $k \geq 2$ is even,}
\endaligned 
\]
\[
\aligned
M_{k,2}
&= \frac{m_k}{2}
\begin{bmatrix}
0 & 0 & 0 & 0 & \cdots & 0 & 0 & 0 & 0 \\
0 & 1 & 0 & 0 & \cdots  & 0 & 0 & -1 & 0 \\
0 & 0 & 1 & 0 & \cdots  & 0 & -1 & 0 & 0 \\
\cdots & \cdots & \cdots & \cdots & \cdots & \cdots & \cdots & \cdots & \cdots \\
0 & \cdots & 0 & 1 & 0 & -1 & 0 & \cdots & 0 \\
0 & \cdots & \cdots & 0 & 0 & 0 & \cdots & \cdots & 0 \\
0 & \cdots & 0 & -1 & 0 & 1 & 0 & \cdots & 0 \\
\cdots & \cdots & \cdots  & \cdots & \cdots & \cdots & \cdots & \cdots & \cdots \\
0 & 0 & -1 & 0 & \cdots  & 0 & 1 & 0 & 0 \\
0 & -1 & 0 & 0 & \cdots  & 0 & 0 & 1 & 0 
\end{bmatrix} \quad \text{if $k \geq 3$ is odd,} \\
&= \frac{m_k}{2}
\begin{bmatrix}
0 & 0 & 0 & 0 & \cdots & \cdots & 0 & 0 & 0 & 0 \\
0 & 1 & 0 & 0 & \cdots & \cdots & 0 & 0 & -1 & 0 \\
0 & 0 & 1 & 0 & \cdots & \cdots & 0 & -1 & 0 & 0 \\
\cdots & \cdots & \cdots & \cdots & \cdots & \cdots & \cdots & \cdots & \cdots & \cdots \\
0 & \cdots & 0 & 1 & 0 & 0 & -1 & 0 & \cdots & 0 \\
0 & \cdots & \cdots & 0 & 1 & -1 & 0 & \cdots & \cdots & 0 \\
0 & \cdots & \cdots & 0 & -1 & 1 & 0 & \cdots & \cdots & 0 \\
0 & \cdots & 0 & -1 & 0 & 0 & 1 & 0 & \cdots & 0 \\
\cdots & \cdots & \cdots & \cdots & \cdots & \cdots & \cdots & \cdots & \cdots & \cdots \\
0 & 0 & -1 & 0 & \cdots & \cdots & 0 & 1 & 0 & 0 \\
0 & -1 & 0 & 0 & \cdots & \cdots & 0 & 0 & 1 & 0 
\end{bmatrix} \quad \text{if $k \geq 2$ is even,}
\endaligned 
\]
\[
\aligned
M_{k,3}
&= \frac{1}{2m_k}
\begin{bmatrix}
0 & 0 & 0 & 0 & \cdots & 0 & 0 & 0 & 0 \\
0 & 1 & 0 & 0 & \cdots  & 0 & 0 & 1 & 0 \\
0 & 0 & 1 & 0 & \cdots  & 0 & 1 & 0 & 0 \\
\cdots & \cdots & \cdots & \cdots & \cdots & \cdots & \cdots & \cdots & \cdots \\
0 & \cdots & 0 & 1 & 0 & 1 & 0 & \cdots & 0 \\
0 & \cdots & \cdots & 0 & 2 & 0 & \cdots & \cdots & 0 \\
0 & \cdots & 0 & 1 & 0 & 1 & 0 & \cdots & 0 \\
\cdots & \cdots & \cdots  & \cdots & \cdots & \cdots & \cdots & \cdots & \cdots \\
0 & 0 & 1 & 0 & \cdots  & 0 & 1 & 0 & 0 \\
0 & 1 & 0 & 0 & \cdots  & 0 & 0 & 1 & 0 
\end{bmatrix} \quad \text{if $k \geq 3$ is odd,} \\
&= \frac{1}{2m_k}
\begin{bmatrix}
0 & 0 & 0 & 0 & \cdots & \cdots & 0 & 0 & 0 & 0 \\
0 & 1 & 0 & 0 & \cdots & \cdots & 0 & 0 & 1 & 0 \\
0 & 0 & 1 & 0 & \cdots & \cdots & 0 & 1 & 0 & 0 \\
\cdots & \cdots & \cdots & \cdots & \cdots & \cdots & \cdots & \cdots & \cdots & \cdots \\
0 & \cdots & 0 & 1 & 0 & 0 & 1 & 0 & \cdots & 0 \\
0 & \cdots & \cdots & 0 & 1 & 1 & 0 & \cdots & \cdots & 0 \\
0 & \cdots & \cdots & 0 & 1 & 1 & 0 & \cdots & \cdots & 0 \\
0 & \cdots & 0 & 1 & 0 & 0 & 1 & 0 & \cdots & 0 \\
\cdots & \cdots & \cdots & \cdots & \cdots & \cdots & \cdots & \cdots & \cdots & \cdots \\
0 & 0 & 1 & 0 & \cdots & \cdots & 0 & 1 & 0 & 0 \\
0 & 1 & 0 & 0 & \cdots & \cdots & 0 & 0 & 1 & 0 
\end{bmatrix} \quad \text{if $k \geq 2$ is even,}
\endaligned 
\]
\[
\aligned
M_{k,4}
&= \frac{1}{2}
\begin{bmatrix}
2 & 0 & 0 & 0 & \cdots & 0 & 0 & 0 & -2 \\
0 & 1 & 0 & 0 & \cdots  & 0 & 0 & -1 & 0 \\
0 & 0 & 1 & 0 & \cdots  & 0 & -1 & 0 & 0 \\
\cdots & \cdots & \cdots & \cdots & \cdots & \cdots & \cdots & \cdots & \cdots \\
0 & \cdots & 0 & 1 & 0 & -1 & 0 & \cdots & 0 \\
0 & \cdots & \cdots & 0 & 0 & 0 & \cdots & \cdots & 0 \\
0 & \cdots & 0 & -1 & 0 & 1 & 0 & \cdots & 0 \\
\cdots & \cdots & \cdots  & \cdots & \cdots & \cdots & \cdots & \cdots & \cdots \\
0 & 0 & -1 & 0 & \cdots  & 0 & 1 & 0 & 0 \\
0 & -1 & 0 & 0 & \cdots  & 0 & 0 & 1 & 0 
\end{bmatrix} \quad \text{if $k \geq 3$ is odd,} \\
&= \frac{1}{2}
\begin{bmatrix}
2 & 0 & 0 & 0 & \cdots & \cdots & 0 & 0 & 0 & -2 \\
0 & 1 & 0 & 0 & \cdots & \cdots & 0 & 0 & -1 & 0 \\
0 & 0 & 1 & 0 & \cdots & \cdots & 0 & -1 & 0 & 0 \\
\cdots & \cdots & \cdots & \cdots & \cdots & \cdots & \cdots & \cdots & \cdots & \cdots \\
0 & \cdots & 0 & 1 & 0 & 0 & -1 & 0 & \cdots & 0 \\
0 & \cdots & \cdots & 0 & 1 & -1 & 0 & \cdots & \cdots & 0 \\
0 & \cdots & \cdots & 0 & -1 & 1 & 0 & \cdots & \cdots & 0 \\
0 & \cdots & 0 & -1 & 0 & 0 & 1 & 0 & \cdots & 0 \\
\cdots & \cdots & \cdots & \cdots & \cdots & \cdots & \cdots & \cdots & \cdots & \cdots \\
0 & 0 & -1 & 0 & \cdots & \cdots & 0 & 1 & 0 & 0 \\
0 & -1 & 0 & 0 & \cdots & \cdots & 0 & 0 & 1 & 0 
\end{bmatrix} \quad \text{if $k \geq 2$ is even.}
\endaligned 
\]
Then, we find that 
\[
P_k(m_k)^{-1}Q_k=
\begin{bmatrix}
M_{k,1} & M_{k,2} \\ M_{k,3} & M_{k,4}   
\end{bmatrix} \quad (2k+2) \times (2k+4)
\]
for every $0 \leq k \leq 2g$. 
This formula of $P_k(m_k)^{-1}Q_{k}$ shows that 
\[
\aligned
(P_{2g-n}(m_{2g-n})^{-1}&Q_{2g-n}v_g(n-1))[1] \\ 
& =v_g(n-1)[1]+v_g(n-1)[2g-n+2], \\
(P_{2g-n}(m_{2g-n})^{-1}&Q_{2g-n}v_g(n-1))[2g-n+2] \\ 
& =v_g(n-1)[2g-n+3]-v_g(n-1)[4g-2n+4]
\endaligned
\]
for every $1 \leq n \leq 2g$. These equalities imply the lemma. 
\end{proof}
By Lemma \ref{lem_203}, we can define $v_g(n)$ by 
taking 
$\tilde{v}_g(n)=P_{2g-n}(m_{2g-n})^{-1}Q_{2g-n} \, v_g(n-1)$ 
for $v_g(n-1)$ 
and then substituting the value $\tilde{v}_g(n)[1]/\tilde{v}_g(n)[2g-n+2]$ into $\tilde{v}_g(n)$. 

Anyway, staring from the initial vector $v_g(0)$, 
$2g$ vectors 
\[
v_g(1),\, v_g(2), \cdots, v_g(2g-1), \, v_g(2g) 
\]
are defined by using $P_k(m_k)$ and $Q_k$. 
They have entries in $\Q(x_0,\cdots,x_g,y_0,\cdots,y_g)$, 
in other words, their entires are functions of the initial vector $v_g(0)$. 

%
%
\subsection{The first result} \label{section_2_2}
%
%
Let $\underline{c}=(c_0,\cdots,c_g)$ be $(g+1)$ indeterminate elements 
and let $q$ be a real variable. We take the column vector 
\begin{equation} \label{def_v0}
\aligned
v_g(0) = 
\begin{bmatrix}
{\mathbf a}_g(0) \\
{\mathbf b}_g(0)
\end{bmatrix}, \quad 
{\mathbf a}_g(0) =
\begin{bmatrix}
c_0 \\ c_1 \\ \vdots \\ c_{g-1} \\ c_g \\ c_{g-1} \\ \vdots \\ c_1 \\ c_0
\end{bmatrix}, \quad  
{\mathbf b}_g(0) = 
\begin{bmatrix}
c_0 \log(q^g) \\ c_1 \log(q^{g-1}) \\ \vdots \\ c_{g-1} \log q \\ 0 \\ -c_{g-1} \log q \\ \vdots \\ -c_1 \log(q^{g-1}) \\ - c_0  \log(q^{g})
\end{bmatrix}
\endaligned
\end{equation}
of length $(4g+2)$ 
as the initial vector of the system consisting of \eqref{def_m1} and \eqref{def_v1}.  
Then $2g$ column vectors 
\[
v_g(1),\, v_g(2), \cdots, v_g(2g-1), \, v_g(2g)
\]
are defined with enties in $\Q(c_1,\cdots,c_g,\log q)$ 
and $m_{2g-n}(\underline{c}\,;\log q) \in \Q(c_1,\cdots,c_g,\log q)$ 
are defined by
\begin{equation}\label{def_m2}
\aligned
m_{2g-n}(\underline{c}\,;\log q):
& = \frac{v_g(n-1)[1]+v_g(n-1)[2g-n+2]}{v_g(n-1)[2g-n+3]-v_g(n-1)[4g-2n+4]} \\
& =\frac{(P_{2g-n}(m_{2g-n})^{-1}Q_{2g-n} \, v_g(n-1))[1]}{(P_{2g-n}(m_{2g-n})^{-1}Q_{2g-n} \, v_g(n-1))[2g-n+2]} \\
& = \frac{v_g(n)[1]}{v_g(n)[2g-n+2]}
\endaligned
\end{equation}
for $1 \leq n \leq 2g$ 
according to \eqref{def_m1} and \eqref{def_v1}, 
where the second equality follows from Lemma \ref{lem_203} 
and the third equality is definition \eqref{def_v1}. 
\smallskip

Now the first result is stated as follows. 
 
\begin{theorem}\label{thm_1}
Let $g \geq 1$ and $q>1$. Let $P_g(x)$ be a self-reciprocal polynomial \eqref{def_Pg} 
of degree $2g$ with real coefficients $\underline{c}=(c_0,\cdots,c_g)$. 
Define $2g$ numbers $m_{2g-n}(\underline{c}\,;\log q)$ $(1\leq n \leq 2g)$ 
by substituting the numerical vector $\underline{c}$ into  \eqref{def_m2}. 
Then all zeros of $P_g(x)$ lie on $T$ and simple
if and only if 
\[
m_{2g-n}(\underline{c}\,;\log q) >0 \quad \text{and} \quad m_{2g-n}(\underline{c}\,;\log q)^{-1} >0
\]
for every $1 \leq n \leq 2g$. 
\end{theorem}
\begin{proof} 
This is proved in Section \ref{section_6}
\end{proof}
\begin{remark} 
The necessity part of Theorem \ref{thm_1} shows that 
$v_g(n)$ and $m_{2g-n}(\underline{c},\,;\log q)$ 
are well-defined for the special initial vector \eqref{def_v0} as well as in the proof of Lemma \ref{lem_202}.
\end{remark}
\begin{remark} 
The criterion of Theorem \ref{thm_1} can be written as    
$m_{2g-n}(\underline{c}\,;\log q)>0$ and $m_{2g-n}(\underline{c}\,;\log q)^{-1}\not=0$. 
The reason why we state it as in Theorem \ref{thm_1} is for the convenience of later subjects.
See \eqref{def_mq}, \eqref{def_Hq} and Corollary \ref{cor_1} below. 
It is similar about Theorem \ref{thm_2}, \ref{thm_3}, \ref{thm_4} below.
\end{remark}
\begin{remark} The strict inequalities $m_{2g-n}(\underline{c}\,;\log q)>0$ and 
$m_{2g-n}(\underline{c}\,;\log q)^{-1}>0$ are essential. 
In fact, the self-reciprocal polynomial $P_2(x)=4(x^4+1)-16(x^3+x)+23x^2 =(2x^2-3x+2)(2x^2-5x+2)$ 
has a zero outside $T$ but $m_{3}=(2\log q)^{-1}>0$, $m_2=m_1=0$, $m_{0}=(126\log q)^{-1}>0$ 
for any $q>1$. 
\end{remark}

The criterion of Theorem \ref{thm_1} is independent of the choice of $q>1$. 
In order to clarify such thing, we modify the statement of Theorem \ref{thm_1} as follows. 

In addition to definition \eqref{def_m2}, 
we set the convention 
\begin{equation} \label{def_m2_2}
m_{2g}(\underline{c}\,;\log q):=\frac{1}{g \log q}. 
\end{equation}
We define $R_n(\underline{c})=R_{g,n}(\underline{c})$ for $0 \leq n \leq 2g$ by 
\[
R_0(\underline{c}):=1
\]
and 
\begin{equation} \label{def_R}
R_{n}(\underline{c})
:=
\begin{cases}
\displaystyle{\prod_{j=0}^{J} \frac{m_{2g-(2j+1)}(\underline{c}\,;\log q)}{m_{2g-2j}(\underline{c}\,;\log q)}} & \text{if $n=2J+1\geq 1$}, \\
\displaystyle{\prod_{j=0}^{J} \frac{m_{2g-(2j+2)}(\underline{c}\,;\log q)}{m_{2g-(2j+1)}(\underline{c}\,;\log q)}} & \text{if $n=2J+2 \geq 2$}.
\end{cases}
\end{equation}
By the definition, we have
\begin{equation} \label{rel_mR}
m_{2g-n}(\underline{c}\,;\log q)
=\frac{1}{g\log q}\,R_{n-1}(\underline{c})R_n(\underline{c}) 
\end{equation}
for every $1 \leq n \leq 2g$. 

\begin{lemma} 
We have $R_n(\underline{c}) \in \Q(c_0,\cdots,c_g)$ for every $0 \leq n \leq 2g$. 
\end{lemma}
\begin{proof} 
Put $F=\Q(c_0,\cdots,c_g)$. 
At first, we show that $v_g(n)[k] \in F$ for $1 \leq k \leq 2g-n+1$ 
and  $(\log q)^{-1}v_g(n)[k] \in F$ for $2g-n+2 \leq k \leq 4g-2n+2$. 
It is clear for $v_g(0)$ by definition \eqref{def_v0}. 
Assume that it holds for $v_g(n-1)$. 
Then, by definition \eqref{def_m2}, 
we have $m_{2g-n}(\underline{c}\,;\log q)=(\log q)^{-1}\mu_{2g-n}(\underline{c})$ 
for some $\mu_{2g-n}(\underline{c}) \in F$. 
By applying the formula of $P_{k}(m_{k})^{-1}Q_{k}$ in the proof of Lemma \ref{lem_203} to $k=2g-n$, 
we obtain 
$v_g(n)[k] \in F$ for $1 \leq k \leq 2g-n+1$ 
and $(\log q)^{-1}v_g(n)[k] \in F$ for $2g-n+2 \leq k \leq 4g-2n+2$. 
Hence $(\log q)m_{2g-n}(\underline{c}\,;\log q) \in F$ 
for every $1 \leq n \leq 2g$ by induction. 

By definition \eqref{def_R}, 
$R_n(\underline{c})$ has the same number of $m_{2g-j}(\underline{c}\,;\log q)$'s in the denominator and the numerator. 
Hence $R_n(\underline{c}) \in F$ 
for every $0 \leq n \leq 2n$.
\end{proof}

Let $\underline{c} \in \R^{g+1}$ and $q>1$. 
If $R_n(\underline{c})>0$  for every $1 \leq n \leq 2g$, 
then $m_{2g-n}(\underline{c}\,;\log q)>0$ for every $1 \leq n \leq 2g$ by \eqref{rel_mR}. 
Conversely, if $m_{2g-n}(\underline{c}\,;\log q)>0$ for every $1 \leq n \leq 2g$, 
then $R_{n-1}(\underline{c})$ and $R_{n}(\underline{c})$ must have the same sign for every $1 \leq n \leq 2g$ by \eqref{rel_mR}. 
However, $R_0(\underline{c})=1>0$ by the definition. 
Hence $R_n(\underline{c})>0$  for every $1 \leq n \leq 2g$. 
As a consequence, Theorem \ref{thm_1} is equivalent to the following statement. 

\begin{theorem} \label{thm_2}
Let $P_g(x)$ be a self-reciprocal polynomial \eqref{def_Pg} of degree $2g>0$ 
with real coefficients $\underline{c}=(c_0,\cdots,c_g)$. 
Define $2g$ numbers $R_n(\underline{c})$ by substituting $\underline{c}$ into \eqref{def_R}.  
Then all zeros of $P_g(x)$ lie on $T$ and simple if and only if 
\[
R_{n}(\underline{c}) >0 \quad \text{and} \quad R_{n}(\underline{c})^{-1}>0
\]
for every $1 \leq n \leq 2g$. 
\end{theorem}

\noindent
For small positive integer $g$, we can calculate $R_n(\underline{c})$ by hand according to the definition. 
Because $R_1(\underline{c})=1$ by the definition, 
interesting values are $R_2(\underline{c}),\cdots,R_{2g}(\underline{c})$. 
For example, we have the following small table:

$\bullet$ $g=1$, $\displaystyle{
R_{2}(c_0,c_1) = \frac{2c_0 + c_1}{2c_0 - c_1}}$. 

$\bullet$ $g=2$, $R_{n}=R_{n}(c_0,c_1,c_2)$ ($2 \leq n \leq 4$), 
\[
R_{2} = \frac{4c_0+c_1}{4c_0-c_1}, \quad 
R_{3} = \frac{8c_0^2-2c_1^2+4c_0c_2}{8c_0^2+c_1^2-4c_0c_2}, \quad 
R_{4} = \frac{2c_0+2c_1+c_2}{2c_0-2c_1+c_2}.
\]

$\bullet$ $g=3$, $R_{n}=R_{n}(c_0,c_1,c_2,c_3)$ ($2 \leq n \leq 6$), 
\[
\aligned
R_{2} &= \frac{6c_0+c_1}{6c_0-c_1}, \quad 
R_{3} = \frac{18c_0^2-3c_1^2+6c_0c_2}{18c_0^2+2c_1^2-6c_0c_2}, \\
R_{4} &= \frac{36c_0^3+6c_0^2c_1-c_0c_1^2+4c_1^3-14c_0c_1c_2+c_1^2c_2-4c_0c_2^2+18c_0^2c_3+3c_0c_1c_3}{36c_0^3-6c_0^2c_1-c_0c_1^2-4c_1^3+14c_0c_1c_2+c_1^2c_2-4c_0c_2^2-18c_0^2c_3+3c_0c_1c_3}, \\
R_{5} &= 
(108c_0^4-21c_0^2c_1^2-12c_1^4+108c_0^3c_2+42c_0c_1^2c_2 \\& \qquad -12c_0^2c_2^2+3c_1^2c_2^2-12c_0c_2^3-54c_0^2c_1c_3-6c_1^3c_3+30c_0c_1c_2c_3-27c_0^2c_3^2)/ \\
& \quad\,\, 
(108c_0^4+9c_0^2c_1^2+8c_1^4-108c_0^3c_2-42c_0c_1^2c_2 \\& \qquad  +36c_0^2c_2^2+c_1^2c_2^2-4c_0c_2^3+54c_0^2c_1c_3-4c_1^3c_3+18c_0c_1c_2c_3-27c_0^2c_3^2),
\\
R_{6} &= \frac{2c_0+2c_1+2c_2+c_3}{2c_0-2c_1+2c_2-c_3}.
\endaligned
\]
See also Section \ref{section_8} (1) for another formula of $R_n(\underline{c})$. 

%
%
\subsection{The second result} \label{section_2_3}
%
%
\noindent 
In order to deal with the case that $P_g(x)$ may have a multiple zero on $T$, 
we consider the following variant of $m_{2g-n}(\underline{c}\,;\log q)$ and $R_n(\underline{c})$.   
\medskip

In stead of the vector \eqref{def_v0}, we take the column vector 
\begin{equation} \label{def_v0_2}
v_g(0) 
=
\begin{bmatrix}
{\mathbf a}_{g,\omega}(0) \\
{\mathbf b}_{g,\omega}(0)
\end{bmatrix}, \quad 
{\mathbf a}_{g,\omega}(0)={\mathbf b}_{g,\omega}(0)
=
\begin{bmatrix}
c_0 \, q^{g \omega} \\
c_1 \, q^{(g-1)\omega} \\
\vdots \\
c_{g-1} \, q^{\omega} \\
c_g \\  
c_{g-1} \, q^{-\omega} \\ 
\vdots \\ 
c_0 \, q^{-g\omega}
\end{bmatrix}  
\end{equation}
of length $(4g+2)$ as the initial vector 
of the system consisting of \eqref{def_m1} and \eqref{def_v1}. Then $2g$ vectors 
\[
v_g(1),\, v_g(2), \cdots, v_g(2g-1), \, v_g(2g)
\]
are rational functions of $\underline{c}=(c_0,\cdots,c_g)$ and $q^\omega$ over $\Q$. We define
\begin{equation}\label{def_m3}
\aligned
m_{2g-n}(\underline{c}\,;q^\omega):
& = \frac{v_g(n-1)[1]+v_g(n-1)[2g-n+2]}{v_g(n-1)[2g-n+3]-v_g(n-1)[4g-2n+4]} 
\endaligned
\end{equation}
for $1 \leq n \leq 2g$ and  
\begin{equation} \label{def_m3_2}
m_{2g}(\underline{c}\,;q^\omega):=\frac{q^{g\omega}+q^{-g\omega}}{q^{g\omega}-q^{-g\omega}} 
\end{equation}
as well as \eqref{def_m2} and \eqref{def_m2_2}. 
Further, we define $R_n(\underline{c}\,;q^\omega)$ for $0 \leq n \leq 2g$ by 
\[
R_0(\underline{c}\,;q^\omega):=1
\]
and 
\begin{equation} \label{def_R2}
R_n(\underline{c}\,;q^\omega) := 
\begin{cases}
\displaystyle{\prod_{j=0}^{J} \frac{m_{2g-(2j+1)}(\underline{c}\,;q^\omega)}{m_{2g-2j}(\underline{c}\,;q^\omega)}} & \text{if $n=2J+1 \geq 1$}, \\
\displaystyle{\prod_{j=0}^{J} \frac{m_{2g-(2j+2)}(\underline{c}\,;q^\omega)}{m_{2g-(2j+1)}(\underline{c}\,;q^\omega)}} & \text{if $n=2J+2 \geq 2$}.
\end{cases}
\end{equation}
Then, we obtain the following second results.
\begin{theorem} \label{thm_3}
Let $q>1$. Let $P_g(x)$ be a self-reciprocal polynomial \eqref{def_Pg} of degree $2g>0$ 
with real coefficients $\underline{c}=(c_0,\cdots,c_g)$. 
Define $2g$ numbers $m_{2g-n}(\underline{c}\,;q^{\omega})$ by substituting $\underline{c}$ into \eqref{def_m3}.
Then, all zeros of $P_g(x)$ lie on $T$ (allowing multiple zeros) 
if and only if 
\[
m_{2g-n}(\underline{c}\,;q^{\omega}) > 0 \quad \text{and} \quad m_{2g-n}(\underline{c}\,;q^{\omega})^{-1}>0
\]
for every $1 \leq n \leq 2g$ and $\omega>0$. 
\end{theorem}
\begin{theorem} \label{thm_4}
Let $q>1$. Let $P_g(x)$ be a self-reciprocal polynomial \eqref{def_Pg} of degree $2g>0$ 
with real coefficients $\underline{c}=(c_0,\cdots,c_g)$. 
Define $2g$ numbers $R_n(\underline{c}\,;q^{\omega})$ by substituting $\underline{c}$ into \eqref{def_R2}.
Then, all zeros of $P_g(x)$ lie on $T$ (allowing multiple zeros) 
if and only if 
\[
R_{n}(\underline{c}\,;q^{\omega}) >0  \quad \text{and} \quad R_{n}(\underline{c}\,;q^{\omega})^{-1}>0
\]
for every $1 \leq n \leq 2g$ and $\omega>0$. 
\end{theorem}

The rational functions of \eqref{def_R} and \eqref{def_R2} 
have the following simple relation. 

\begin{theorem} \label{thm_5}
Let $R_{n}(\underline{c}\,;q^\omega)$ and $R_{n}(\underline{c}\,;q^\omega)$ 
be in \eqref{def_R} and \eqref{def_R2}, respectively. Then 
\[
\lim_{q^\omega \to 1^+} R_n(\underline{c}\,;q^\omega)
= R_{n}(\underline{c})
\]
as a rational function of $\underline{c}=(c_0,\cdots,c_g)$ over $\Q$.  
Suppose that all zeros of a self-reciprocal polynomial \eqref{def_Pg} with real coefficients $\underline{c}=(c_0,\cdots,c_g)$ 
lie on $T$ and simple. Then we have 
\[
R_n(\underline{c}\,;q^\omega)
= R_{n}(\underline{c})+O(\log q^\omega) \quad \text{as} \quad q^\omega \to 1^+, 
\]
and 
\[
\aligned
m_{2g-n}(\underline{c}\,;q^\omega) 
&=\frac{1}{g \log q^\omega}\Bigl( R_{n-1}(\underline{c})R_n(\underline{c}) + O(\log q^\omega) \Bigr)
+O(\log q^\omega) \quad \text{as} \quad q^\omega \to 1^+,\\ 
& =
\frac{1}{\omega}\Bigl( m_{2g-n}(\underline{c}\,;\log q)+ O(\omega/g) \Bigr)
+O(\omega \log q) \quad \text{as} \quad \omega \to 0^+,
\endaligned
\]
where implied constants depend only on  $\underline{c}$. 
\end{theorem}

Theorem \ref{thm_3}, \ref{thm_4}, and \ref{thm_5} 
are proved in Section \ref{section_7} 
together with Theorem \ref{thm_7} and Corollary \ref{cor_2} below. 

%
%
\section{Canonical systems} \label{section_3}
%
%
The positivity of values $m_{2g-n}(\underline{c}\,;\log q)$, $m_{2g-n}(\underline{c}\,;q^\omega)$,  
$R_n (\underline{c})$, $R_n (\underline{c}\,;q^\omega)$ 
for a numerical vector $\underline{c}$ attached to a self-reciprocal polynomial of degree $2g$ 
having only (simple) zeros on $T$ 
can be understood from the viewpoint of canonical systems
of linear differential equations (in the sense of de Branges).
In fact, ideas of constructions of $m_{2g-n}(\underline{c}\,;\log q)$ and $m_{2g-n}(\underline{c}\,;q^\omega)$ 
in Section \ref{section_2}
are coming from the theory of canonical systems. 

In this section, we review the theory of canonical systems of linear differential equations 
and the theory of entire functions of the Hermite--Biehler class  
according to de Branges \cite{deBranges68}, Dym~\cite{Dym70}, Levin~\cite{Levin80}, Remling \cite{Remling02}, 
and Lagarias \cite{Lagarias06, Lagarias09},   
in order to understand Theorem \ref{thm_1} and \ref{thm_3} 
in terms of these theories. 
We often use the notation 
$F^\sharp(z)=\overline{F(\bar{z})}$ 
for an entire function $F(z)$. 
An entire function $F(z)$ is called real or a real entire function if $F(\R) \subset \R$, 
or equivalently, $F^\sharp(z)=F(z)$.
We denote by $\lim_{x \to x_0^+}$ and $\lim_{x \to x_0^-}$  
the right-sided limit and the left-sided limit at $x=x_0$, respectively. 

\begin{definition} \label{def_301}
Let $H(a)$ be a $2\times 2$ matrix-valued function 
defined almost everywhere on a finite interval $I=[a_1,a_0)$ $(-\infty<a_1<a_0 < \infty)$. 
A family of linear differential equations on $I$ 
of the form 
\begin{equation} \label{can_0}
-a\frac{\partial }{\partial a}
\begin{bmatrix}
A(a,z) \\ B(a,z)
\end{bmatrix}
= z 
\begin{bmatrix}
0 & -1 \\ 1 & 0
\end{bmatrix}
H(a)
\begin{bmatrix}
A(a,z) \\ B(a,z)
\end{bmatrix}, 
\quad 
\lim_{a \to a_0^-}
\begin{bmatrix}
A(a,z) \\ B(a,z)
\end{bmatrix}
= 
\begin{bmatrix}
1 \\ 0
\end{bmatrix}
\end{equation}
parametrized by $z \in \C$ 
is called a (two-dimensional) canonical system 
if 
\begin{enumerate}
\item[(H1)] $H(a)$ is a positive semidefinite symmetric matrix for almost every $a \in I$, 
\item[(H2)] $H(a)\not\equiv 0$ on any open subset of $I$ with positive Lebesgue measure,  
\item[(H3)] $H(a)$ is locally integrable on $I$.  
\end{enumerate}
For a canonical system, the matrix-valued function $H(a)$ is called its Hamiltonian. 
\end{definition}

Usually, canonical systems are defined by using the additive derivative 
$\partial/\partial t$ for $t=\log(a_0/a)$. 
However, we use the multiplicative derivative $-a(\partial/\partial a)$ as above 
for the convenience of descriptions of results in the paper. 

Canonical systems are closely related to 
entire functions of the Hermite--Biehler class. 

\begin{definition}[Hermite--Biehler class]
An entire function $E(z)$ is said to be a function of class {\rm HB} 
if it satisfies the condition  
\begin{equation} \label{HB}
|E^\sharp(z)| < |E(z)| \quad \text{for every} \quad {\rm Im}\,z > 0
\end{equation}
and has no real zeros. 
On the other hand, 
an entire function $E(z)$ is said to be a function of class $\overline{\rm HB}$
if it satisfies the condition  
\begin{equation} \label{wHB}
|E^\sharp(z)| \leq |E(z)| \quad \text{for every} \quad {\rm Im}\,z > 0.
\end{equation}
and has no zeros in the upper half-plane ${\rm Im}\, z>0$.
\end{definition}
\begin{remark}
This definition of class HB is equivalent to the definition of Levin~\cite[\S1 of Chap. VII]{Levin80} 
if we replace the word ``the upper half-plane''  by ``the lower half-plane'', 
because \eqref{HB} implies that $E(z)$ has no zeros in the upper half-plane ${\rm Im}\, z>0$. 
We adopted this definition for the convenience of using of 
the theory of canonical systems via the theory of de Branges spaces.
\end{remark}
The following result for a function of class {\rm HB} 
is used often in the later sections. 
\begin{proposition} \label{lem_303}
Let $E(z)$ be an entire function of finite order. 
Put 
 \[
A(z):=\frac{1}{2}(E(z)+E^\sharp(z)) , \quad B(z):=\frac{i}{2}(E(z)-E^\sharp(z)). 
\]
Then $E(z)$ is a function of class {\rm HB} 
if and only if
$E(z)$ has no zeros in the upper half-plane ${\rm Im}\, z>0$, 
(real) entire functions $A(z)$ and $B(z)$ have only simple real zero, 
and zeros of $A(z)$ and $B(z)$ interlace.
 On the other hand, 
the function $E(z)$ is a function of class $\overline{\rm HB}$ 
if and only if $E(z)$ is a product of a real entire function $E_0(z)$ 
having only real zeros and a function $E_1(z)$ of class {\rm HB}. 
\end{proposition}
\begin{proof}
See Levin~\cite[Chap. VII, Theorem 3, Theorem 5, the latter half of p.313]{Levin80}.
\end{proof}

There are two important results of de Branges 
that relate a canonical system with an entire function of class HB. 
Roughly, if $(A(a,z),B(a,z))$ is a solution of a canonical system, 
then $E(a,z):=A(a,z)-iB(a,z)$ is a function of class HB for every $[a_1,a_0)$ 
(see \cite[Theorem 41]{deBranges68}, and also \cite[Section 2]{Dym70} for details). 
Conversely, if $E(z)$ is a function of class HB normalized as $E(0)=1$, 
then there exists a canonical system on some interval $[a_1,a_0)$ 
such that $E(z)=A(a_1,z)-iB(a_1,z)$ for the solution $(A(a,z),B(a,z))$ of the canonical system 
(see \cite[Theorem 40]{deBranges68}, and also \cite[Theorem 7.3]{Remling02}, \cite[pp.70--71]{Lagarias06} for more details). 
Therefore, a function of class HB and a Hamiltonian of a canonical system 
correspond each other. 
However, in general, an explicit construction of a Hamiltonian is quite difficult 
when we start from a function of class HB, because it is a kind of inverse scattering problem.  

%
%
\section{Results II: \\ Differential equations attached to self-reciprocal polynomials} \label{section_4} 
%
%
In this section, 
we construct two kinds of systems of ordinary linear differential equations attached to a self-reciprocal polynomial 
so that the first one (resp. the second one) 
is a (two dimensional) canonical system 
if and only if all zeros of the polynomial lie on $T$ and simple (resp. lie on $T$).

%
%
\subsection{Systems of the first kind} \label{section_4_1}
%
%
Fix a real number $q>1$ arbitrary. 
For a self-reciprocal polynomial $P_{g}(x)$ of \eqref{def_Pg} with real coefficients $\underline{c}=(c_0,\cdots,c_g)$, we define
\begin{equation} \label{def_A}
A_q(z) := q^{-giz}P_g(q^{iz}) = \sum_{k=0}^{g-1} c_k\Bigl(q^{(g-k)iz} + q^{-(g-k)iz}\Bigr) + c_g.
\end{equation}
By the definition, all zeros of $P_g(x)$ lie on $T$ if and only if $A_q(z)$ has only real zeros. 
The self-reciprocal condition $P_g(x)=x^{2g}P_g(1/x)$ implies the functional equation 
$ A_q(z)=A_q(-z) $
and the realness of coefficients of $P_g(x)$ implies
$ A_q^\sharp(z)=A_q(z). $
Hence $A_q(z)$ is an even real entire function of exponential type.  
Further, we define 
\begin{equation} \label{def_B}
B_q(z):=-\frac{d}{dz}A_q(z)
\end{equation}
and 
\begin{equation} \label{def_E}
E_q(z) := A_q(z) - i B_q(z).
\end{equation}
Then $B_q(z)$ is a real entire function and  
\begin{equation}
B_q(z) = -i(\log q)\sum_{k=0}^{g-1} (g-k) c_k\Bigl(q^{(g-k)iz} - q^{-(g-k)iz}\Bigr).
\end{equation}
Moreover, we have
\begin{equation} \label{def_E_2}
E_q^\sharp(z) =  A_q(z) + i B_q(z)
\end{equation}
and 
\[
A_q(z) = \frac{1}{2}(E_q(z)+E_q^\sharp(z)), \quad B_q(z)=\frac{i}{2}(E_q(z)-E_q^\sharp(z)).
\]

\begin{lemma} \label{lem_401}
Let $E_q(z)$, $A_q(z)$, $B_q(z)$ be as above. Then 
\begin{enumerate}
\item $E_q(z)$ satisfies condition \eqref{HB}
if and only if $A_q(z)$ has only real zeros. 
\item $E_q(z)$ is a function of class {\rm HB}
if and only if $A_q(z)$ has only simple real zeros. 
\end{enumerate}
\end{lemma}
By this lemma, if $A_q(z)$ has only real zeros, 
then $E_q(z)$ is a function of class $\overline{\rm HB}$ at least 
and $B_q(z)$ has only real zeros, 
but $A_q(z)$ and $B_q(z)$ may have a common (real) zero. 
\begin{proof} 
(1) Assume that $E_q(z)$ satisfies \eqref{HB}. 
Then it implies that $A_q(z) \not=0$ for ${\rm Im}\, z>0$,  
Further, $A_q(z) \not=0$ for ${\rm Im}\, z<0$ 
by the functional equation $A_q(z)=A_q(-z)$. 
Hence all zeros of $A_q(z)$ lie on the real line. 
Conversely, assume that all zeros of $A_q(z)$ are real. 
Then $A_q(z)$ has the factorization 
\begin{equation} \label{H-factorization}
A_q(z) = C \lim_{R \to \infty}\prod_{|\rho| \leq R}\left(1-\frac{z}{\rho} \right) \quad (C,\,\rho \in \R), 
\end{equation}
because $A_q(z)$ is real, even and of exponential type. 
Therefore, 
\[
{\rm Re}\left( i\frac{A_q^\prime(z)}{A_q(z)} \right) 
= {\rm Re}\left(  \sum_{\rho} \frac{i(x-\rho) + y}{|z-\rho|^2} \right) 
= \sum_{\rho} \frac{y}{|z-\rho|^2} \quad (z=x+iy).
\]
Hence, for ${\rm Im}\, z>0$, 
\[
|E_q(z)|=|A_q(z)|\left|1+i\frac{A_q^\prime(z)}{A_q(z)} \right|>|A_q(z)|\left|1-i\frac{A_q^\prime(z)}{A_q(z)} \right|=|E_q^\sharp(z)|.
\]
\noindent
(2) Suppose that $E_q(z)$ is a function of class HB, that is, 
$E_q(z)$ satisfies \eqref{HB} and has no real zeros. 
Then, $A_q(z)$ has only real zeros by (1). 
If $A_q(z)$ has a multiple real zero, 
then $A_q(z)$ and $B_q(z)=-A_q^\prime(z)$ have a common real zero.  
Thus $E_q(z)=A_q(z)-iB_q(z)$ has a real zero. 
It is a contradiction. 
Hence $A_q(z)$ has only simple real zeros. 
The converse assertion follows from (1) and definition \eqref{def_B}. 
\end{proof}

As mentioned in Section \ref{section_3}, 
functions of class HB and canonical systems correspond each other. 
Therefore, by Lemma \ref{lem_401}, 
there must exists a canonical system attached to $E_q(z)$ 
if and only if all zeros of $P_g(x)$ lie on $T$ and simple. 
Now we construct a system of linear differential equation 
so that it is a canonical system if and only if all zeros of $P_g(x)$ lie on $T$ and simple.  
\medskip

Let $\underline{c}=(c_0,\cdots,c_g)$ be $2g$ indeterminate elements 
and let $q>1$ be a real number.  By using $m_{2g-n}(\underline{c}\,;\log q)$ of \eqref{def_m2}, 
we define the $\Q(c_0,\cdots,c_g)$-valued function $m_{q}(a)$ of $a \in [1,q^{g})$ by
\begin{equation} \label{def_mq}
m_{q}(a)  =m_{2g-n}(\underline{c}\,;\log q) \quad \text{if} \quad q^{\frac{n-1}{2}} \leq a < q^{\frac{n}{2}}
\end{equation}
and the $2 \times 2$ matrix valued function $H_q(a)$ of $a \in [1,q^{g})$ by
\begin{equation} \label{def_Hq}
H_q(a) = 
\begin{bmatrix}
m_{q}(a)^{-1} & 0 \\ 0 & m_{q}(a)
\end{bmatrix}.
\end{equation}
%
In addition, we define functions $A_{q}(a,z)$ and $B_{q}(a,z)$ of $(a,z) \in [1,q^{g}) \times \C$ by 
\begin{equation} \label{def_AB}
\begin{bmatrix}
A_{q}(a,z) \\
B_{q}(a,z)
\end{bmatrix} 
= \frac{1}{2}
\begin{bmatrix}
1 & 0 \\
0 & -i 
\end{bmatrix}
T_n(a,z)
v_g(n)
\quad \text{if $q^{\frac{n-1}{2}} \leq a < q^{\frac{n}{2}}$}, 
\end{equation}
where $T_n(a,z)$ is the $2 \times (4g-2n+2)$ matrix valued function
\begin{equation}\label{def_T}
\begin{bmatrix}
c_g(a,z) & c_{g-1}(a,z) & \cdots & c_{-g+n}(a,z) & 0 & 0 & \cdots & 0 \\
0 & 0 & \cdots & 0 & s_g(a,z) & s_{g-1}(a,z) & \cdots & s_{-g+n}(a,z)
\end{bmatrix}
\end{equation}
with 
\begin{equation} \label{def_cs}
c_k(a,z):=((q^k/a)^{iz}+(q^k/a)^{-iz}), \quad s_k(a,z):=((q^k/a)^{iz}-(q^k/a)^{-iz}).
\end{equation}

Then the result of this section is stated as follows. 
\begin{theorem} \label{thm_6} 
Let $\underline{c}=(c_0,\cdots,c_g)$ be real coefficients 
of a self-reciprocal polynomial $P_g(x)$ of degree $2g$ in \eqref{def_Pg} 
and let $q>1$ be a real number. 
Define $E_q(z)$, $A_q(z)$, $B_q(z)$ by \eqref{def_E}, \eqref{def_A}, \eqref{def_B}, respectively. 
Define $H_q(a)$, $A_{q}(a,z)$ and $B_{q}(a,z)$ 
by substituting the numerical vector $\underline{c}$ into \eqref{def_Hq}, \eqref{def_AB}, respectively.  
Define $2g$ numbers $m_{2g-n}(\underline{c}\,;\log q)$ $(1\leq n \leq 2g)$ 
by substituting $\underline{c}$ into  \eqref{def_m2}.
Then we have 
\begin{enumerate}
\item[(1)] $m_{2g-1}(\underline{c}\,;\log q)>0$ and $m_{2g-1}(\underline{c}\,;\log q)^{-1}>0$,  
\item[(2)] $A_q(a,z)$ and $B_q(a,z)$ are defined on $[1,q^{1/2})$ and 
\[
A_q(1,z)=A_q(z), \quad B_q(1,z)=B_q(z).
\] 
\end{enumerate}
Let $1 \leq n_0 \leq 2g$. Suppose that $m_{2g-n}(\underline{c}\,;\log q)\not=0$ and $m_{2g-n}(\underline{c}\,;\log q)^{-1}\not=0$ 
for every $1 \leq n \leq n_0$. Then we have 
\begin{enumerate}
\item[(3)] $A_{q}(a,z)$ and $B_{q}(a,z)$ are defined and continuous on $[1,q^{n_0/2})$  with respect to $a$, 
\item[(4)] the left-sided limit $\lim_{a \to (q^{n_0/2})^-}(A_{q}(a,z), B_{q}(a,z))$ defines entire functions of $z$, 
\item[(5)] $A_{q}(a,z)$ and $B_{q}(a,z)$ are differentiable functions on $(q^{(n-1)/2},q^{n/2})$ with respect to $a$ for every $1 \leq n \leq n_0$, 
\item[(6)] $H_q(a)$, $A_{q}(a,z)$ and $B_{q}(a,z)$ satisfy the system
\begin{equation} \label{system_1}
-a\frac{\partial}{\partial a}
\begin{bmatrix}
A_{q}(a,z) \\ B_{q}(a,z)
\end{bmatrix}
= z 
\begin{bmatrix}
0 & -1 \\ 1 & 0
\end{bmatrix}
H_q(a)
\begin{bmatrix}
A_{q}(a,z) \\ B_{q}(a,z)
\end{bmatrix} \quad (z \in \C)
\end{equation}
for $a \in [1,q^{n_0}/2)$,
\end{enumerate}
Suppose that $A_q(a,z)$ and $B_q(a,z)$ are defined on $[q^{g-1/2},q^{g})$, 
that is, substitution of $\underline{c}$ into $v_g(2g)$ defines a numerical column vector of length $2$. 
Then, we have 
\begin{enumerate}
\item[(7)]  
$\displaystyle{
\lim_{a \to (q^{g})^-}
\begin{bmatrix}
A_{q}(a,z) \\
B_{q}(a,z)
\end{bmatrix}
= A_q(0)
\begin{bmatrix}
1 \\ 0 
\end{bmatrix}
= E_q(0)
\begin{bmatrix}
1 \\ 0
\end{bmatrix}
}$.
\end{enumerate}
\end{theorem}
\begin{proof}
This is proved in Section \ref{section_5}. 
\end{proof}
\begin{corollary} \label{cor_1}
The $2 \times 2$ matrix valued function $H_q(a)$ of \eqref{def_Hq} attached to real coefficients $\underline{c}$ 
of $P_g(x)$ defines a canonical system on $[1,q^g)$ 
if and only if all zeros of $P_g(x)$ lie on $T$ and simple. 
If $H_q(a)$ defines a canonical system on $[1,q^g)$, 
then the pair of functions $(A_q(a,z)/E_q(0)$, $B_q(a,z)/E_q(0))$ is a solution 
of the canonical system. 
\end{corollary}
\begin{proof}
If all zeros of $P_g(x)$ lie on $T$ and simple, 
then $A_q(z)$ has only simple real zeros. 
In particular, $E_q(0)=A_q(0)\not=0$, since $A_q(z)$ is even and $B_q(z)$ is odd. 
Hence, by Theorem \ref{thm_1} and Theorem \ref{thm_6}, 
$H_q(a)$ defines a canonical system on $[1,q^g)$ 
and $(A_q(a,z)/E_q(0)$, $B_q(a,z)/E_q(0))$ is its solution. 
Conversely, if $H_q(a)$ defines a canonical system, 
then $m_{2g-n}(\underline{c}\,;\log q)>0$ 
and $m_{2g-n}(\underline{c}\,;\log q)^{-1}>0$ for every $1 \leq n \leq 2g$
by Definition \ref{def_301}.   
Hence, all zeros of $P_g(x)$ lie on $T$ and simple by Theorem \ref{thm_1}.
\end{proof}

%
%
\subsection{Systems of the second kind} \label{section_4_2}
%
%
By using the function of \eqref{def_A}, we define 
\begin{equation} \label{def_E2}
E_{q,\omega}(z) := A_q(z+i\omega), 
\end{equation}
\begin{equation}
A_{q,\omega}(z) := \frac{1}{2}(E_{q,\omega}(z)+E_{q,\omega}^\sharp(z)), \quad 
B_{q,\omega}(z) := \frac{i}{2}(E_{q,\omega}(z)-E_{q,\omega}^\sharp(z)).
\end{equation}
Then $A_{q,\omega}(z)$ and $B_{q,\omega}(z)$ are real entire functions satisfying 
\begin{equation} \label{def_E2_2}
E_{q,\omega}(z) = A_{q,\omega}(z) - i B_{q,\omega}(z)
\end{equation} 
and 
\begin{equation} \label{def_E2_3}
E_{q,\omega}^\sharp(z) =  A_q(z-i\omega). 
\end{equation}
Therefore, we obtain 
\begin{equation}
\aligned
A_{q,\omega}(z) 
& = \frac{1}{2}\left(A_q(z+i\omega)+A_q(z-i\omega)\right) \\ 
& = \frac{1}{2}\sum_{k=1}^{g} c_{g-k}\Bigl(q^{k \omega} + q^{-k\omega}\Bigr)\Bigl( q^{kiz} + q^{-kiz}\Bigr) + c_g, \\
B_{q,\omega}(z) 
& = \frac{i}{2}\left(A_q(z+i\omega)-A_q(z-i\omega)\right) \\
& = -\frac{i}{2}\sum_{k=1}^{g} c_{g-k}\Bigl(q^{k \omega} - q^{-k\omega}\Bigr)\Bigl( q^{kiz} - q^{-kiz}\Bigr). 
\endaligned
\end{equation}
Further, by using $v_g(n)$ and $m_{2g-n}(\underline{c}\,;q^\omega)$ of Section \ref{section_2_3}, 
we define the function $m_{q,\omega}(a)$ of $a \in [1,q^{g})$ by
\[
m_{q,\omega}(a)
=m_{2g-n}(\underline{c}\,;q^\omega) := \frac{v_g(n)[1]}{v_g(n)[2g-n+2]} 
\quad \text{if} \quad q^{\frac{n-1}{2}} \leq a < q^{\frac{n}{2}}
\]
and define the $2 \times 2$ matrix valued function $H_{q,\omega}(a)$ of $a \in [1,q^{g})$ by
\begin{equation} \label{def_Hq_2}
H_{q,\omega}(a) = 
\begin{bmatrix}
m_{q,\omega}(a)^{-1} & 0 \\ 0 & m_{q,\omega}(a)
\end{bmatrix}.
\end{equation}
In addition, we define functions $A_{q,\omega}(a,z)$ and $B_{q,\omega}(a,z)$ of $a \in [1,q^{g})$ by 
\[
\begin{bmatrix}
A_{q,\omega}(a,z) \\
B_{q,\omega}(a,z)
\end{bmatrix} 
= \frac{1}{2}
\begin{bmatrix}
1 & 0 \\
0 & -i 
\end{bmatrix} T_n(a,z) v_g(n) \quad \text{if} \quad q^{\frac{n-1}{2}} \leq a < q^{\frac{n}{2}}
\]
as well as \eqref{def_AB}, where $T_n(a,z)$ is the $2 \times (4g-2n+2)$ matrix valued function in \eqref{def_T}.

\begin{theorem} \label{thm_7} 
All assertions of Theorem \ref{thm_6} hold  
if we replace 
$E_q(z)$, $A_q(z)$, $B_q(z)$, $H_q(a)$, $A_q(a,z)$, $B_q(a,z)$ 
by  
$E_{q,\omega}(z)$, $A_{q,\omega}(z)$, $B_{q,\omega}(z)$, $H_{q,\omega}(a)$, $A_{q,\omega}(a,z)$, $B_{q,\omega}(a,z)$, 
respectively.  
\end{theorem}
\begin{corollary} \label{cor_2}
The $2 \times 2$ matrix valued function $H_{q,\omega}(a)$ defines a canonical system on $[1,q^g)$ 
for every $\omega>0$ if and only if all zeros of $P_g(x)$ lie on $T$. 
If $H_{q,\omega}(a)$ defines a canonical system, 
then $(A_{q,\omega}(a,z)/E_{q,\omega}(0)$, $B_{q,\omega}(a,z)/E_{q,\omega}(0))$ 
is its solution. 
\end{corollary} 

\noindent
We prove Theorem \ref{thm_7} and Corollary \ref{cor_2} in Section \ref{section_7}.   

%
%
%
\section{Proof of Theorem \ref{thm_6} } \label{section_5}
%
%

\noindent 
{\bf Proof of (1) and (2).} 
By  definition \eqref{def_m2}, 
\begin{equation} \label{m_2g-1}
m_{2g-1}=\frac{v_g(0)[1]+v_g(0)[2g+1]}{v_g(0)[2g+2]-v_g(0)[4g+2]}
=\frac{1}{g\log q}>0,
\end{equation}
since $q>1$. This implies (1). 
By definition \eqref{def_v1}, 
$v_g(1)=P_{2g-1}(m_{2g-1})^{-1}Q_{2g-1}\cdot v_g(0)$.  
Therefore, we have
\[
v_g(1) 
= 
\begin{bmatrix}
{\mathbf a}_g(1) \\
{\mathbf b}_g(1)
\end{bmatrix},\quad 
{\mathbf a}_g(1)
=\begin{bmatrix}
2c_0 \\
c_1(1 + m_{2g-1}\log(q^{g-1})) \\
c_2(1 + m_{2g-1}\log(q^{g-2})) \\
\vdots \\
c_{g-1}(1 + m_{2g-1}\log q)) \\
c_g \\
c_{g-1}(1 - m_{2g-1}\log q)) \\
\vdots \\
c_2(1 - m_{2g-1}\log(q^{g-2})) \\
c_1(1 - m_{2g-1}\log(q^{g-1})) \\
\end{bmatrix}, \quad 
{\mathbf b}_g(1)
=
\begin{bmatrix}
2c_0 \log(q^g) \\
c_1(m_{2g-1}^{-1} + \log(q^{g-1})) \\
c_2(m_{2g-1}^{-1} + \log(q^{g-2})) \\
\vdots \\
c_{g-1}(m_{2g-1}^{-1} + \log q) \\
c_g m_{2g-1}^{-1} \\
c_{g-1}(m_{2g-1}^{-1} - \log q) \\
\vdots \\
c_2(m_{2g-1}^{-1} - \log(q^{g-2})) \\
c_1(m_{2g-1}^{-1} - \log(q^{g-1}))
\end{bmatrix}.
\]
By substituting \eqref{m_2g-1} into this formula of $v_g(1)$, 
we obtain 
\begin{equation} \label{vec_v1}
v_g(1) 
= 
\begin{bmatrix}
{\mathbf a}_g(1) \\
{\mathbf b}_g(1)
\end{bmatrix},\quad
{\mathbf a}_g(1)
=\begin{bmatrix}
2c_0 \\
\frac{2g-1}{g}c_1 \\
\frac{2g-2}{g}c_2 \\
\vdots \\
\frac{g+1}{g}c_{g-1} \\
c_g \\
\frac{g-1}{g}c_{g-1} \\
\vdots \\
\frac{2}{g}c_2 \\
\frac{1}{g}c_1 \\
\end{bmatrix}, \quad 
{\mathbf b}_g(1)
=
\begin{bmatrix}
2g c_0 \log q \\
(2g-1) c_1 \log q\\
(2g-2) c_2 \log q \\
\vdots \\
(g+1)c_{g-1} \log q \\
g c_g \log q \\
(g-1)c_{g-1} \log q \\
\vdots \\
2 c_2  \log q  \\
 c_1 \log q 
\end{bmatrix}.
\end{equation}
By definition \eqref{def_AB}, 
\[
\aligned
\, & 
\begin{bmatrix}
A_{q}(1,z) \\
B_{q}(1,z)
\end{bmatrix} 
= \frac{1}{2}
\begin{bmatrix}
1 & 0 \\
0 & -i 
\end{bmatrix}\\
& \times 
\begin{bmatrix}
c_g(1,z) & c_{g-1}(1,z) & \cdots & c_{-g+1}(1,z) & 0 & 0 & \cdots & 0 \\
0 & 0 & \cdots & 0 & s_g(1,z) & s_{g-1}(1,z) & \cdots & s_{-g+1}(1,z)
\end{bmatrix}
v_g(1).
\endaligned
\]
On the other hand, 
it is easy to see that $c_{-k}(1,z)=c_{k}(1,z)$, $s_{-k}(1,z)=-s_{k}(1,z)$, $c_0(1,z)=2$, $s_0(1,z)=0$ 
by definition \eqref{def_cs}. 
Hence we obtain
\[
\aligned
A_{q}(1,z) 
& = \sum_{k=0}^{g-1} c_k c_{g-k}(1,z) + c_g = A_q(z), \\
B_q(1,z) 
& = -i(\log q)\sum_{k=0}^{g-1} (g-k)c_k s_{g-k}(1,z) = B_q(z)
\endaligned
\]
and complete the proof. \hfill $\Box$  
\medskip

\noindent 
{\bf Proof of (3) and (4).} 
By the assumption and definitions \eqref{def_v1} and \eqref{def_m2}, 
numerical vectors $v_g(1),\cdots,v_g(n_0)$ are well-defined. 
Thus, functions $A_q(a,z)$ and $B_q(a,z)$ are defined on $[1,q^{n_0/2})$ 
and (4) is trivial by definition \eqref{def_AB}. 
Therefore, it is sufficient to prove that $A_q(a,z)$ and $B_q(a,z)$ are continuous 
at $a=q^{(n-1)/2}$ for every $1 \leq n \leq n_0-1$. 
By definition \eqref{def_AB}, the continuity of $A_q(a,z)$ and $B_q(a,z)$ at $a=q^{(n-1)/2}$ 
is equivalent to the equality 
\[
T_n(q^{(n-1)/2},z)
v_g(n)
=
T_{n-1}(q^{(n-1)/2},z)
v_g(n-1),
\]
and this is also equivalent to
\[
T_n(q^{(n-1)/2},z)
P_{2g-n}(m_{2g-n})^{-1}Q_{2g-n}
=
T_{n-1}(q^{(n-1)/2},z),
\]
where $m_k=m_k(\underline{c}\,;\log q)$ 
and $T_n(a,z)$ is in \eqref{def_T}.
By the formula of $P_{2g-n}(m_{2g-n})^{-1}Q_{2g-n}$ in the proof of Lemma \ref{lem_203}, 
the latter equality means 
\[
\aligned
\begin{bmatrix}
c_g(q^{(n-1)/2},z) & \cdots & c_{-g+n}(q^{(n-1)/2},z) 
\end{bmatrix}
&
\begin{bmatrix}
M_{2g-n,1}& M_{2g-n,2}
\end{bmatrix} \\
= & 
\begin{bmatrix}
c_g(q^{(n-1)/2},z) & \cdots & c_{-g+n-1}(q^{(n-1)/2},z) 
\end{bmatrix},
\endaligned
\]
\[
\aligned
\begin{bmatrix}
s_g(q^{(n-1)/2},z) & \cdots & s_{-g+n}(q^{(n-1)/2},z) 
\end{bmatrix}
&
\begin{bmatrix}
M_{2g-n,3}& M_{2g-n,4}
\end{bmatrix} \\
= & 
\begin{bmatrix}
s_g(q^{(n-1)/2},z) & \cdots & s_{-g+n-1}(q^{(n-1)/2},z) 
\end{bmatrix}.
\endaligned
\]
These equalities follow from elementary identities 
\[
\aligned
c_{g-(h+1)}(q^{(n-1)/2},z) &= c_{-g+n+h}(q^{(n-1)/2},z), \\
s_{g-(h+1)}(q^{(n-1)/2},z) &= -s_{-g+n+h}(q^{(n-1)/2},z) 
\endaligned
\]
and the definition of $M_{k,j}$ ($1 \leq j \le 4$) in the proof of Lemma \ref{lem_203}.  
\hfill
 $\Box$  
\medskip 

\noindent 
{\bf Proof of (5).} This is obvious by definition \eqref{def_AB}. \hfill $\Box$  
\medskip

\noindent 
{\bf Proof of (6).} 
Suppose that $q^{(n-1)/2} \leq a < q^n$. 
By obvious equalities 
\[
a\frac{\partial}{\partial a}c_k(a,z)=(-iz)s_k(a,z), \quad 
a\frac{\partial}{\partial a}s_k(a,z)=(-iz)c_k(a,z),
\]
we have
\[
\aligned
\, & a\frac{\partial}{\partial a}
\begin{bmatrix}
A_{q}(a,z) \\
B_{q}(a,z)
\end{bmatrix} = \frac{z}{2}
\begin{bmatrix}
-i & 0 \\
0 & -1 
\end{bmatrix}\\
& \times 
\begin{bmatrix}
s_g(a,z) & s_{g-1}(a,z) & \cdots & s_{-g+n}(a,z) & 0 & 0 & \cdots & 0 \\
0 & 0 & \cdots & 0 & c_g(a,z) & c_{g-1}(a,z) & \cdots & c_{-g+n}(a,z)
\end{bmatrix}
v_g(n).
\endaligned
\]
By definition \eqref{def_v1} of $v_g(n)$, we have $P_{2g-n}(m_{2g-n})v_g(n)=Q_{2g-n}v_g(n-1)$. 
By looking at the $(2g-n)$-th row from the bottom, we have
\[
v_g(n)[j]= m_{2g-n}\cdot v_g(n)[j+2g+1-n] \quad \text{for} \quad  1\leq j \leq 2g-n.
\] 
This equality implies that 
\[
\aligned 
\, &
\begin{bmatrix}
s_g(a,z) & s_{g-1}(a,z) & \cdots & s_{-g+n}(a,z) & 0 & 0 & \cdots & 0 \\
0 & 0 & \cdots & 0 & c_g(a,z) & c_{g-1}(a,z) & \cdots & c_{-g+n}(a,z)
\end{bmatrix}
v_g(n) \\
& = 
\begin{bmatrix}
0 & 1 \\
1 & 0 
\end{bmatrix}
\begin{bmatrix}
m_{2g-n}^{-1} & 0 \\
0 & m_{2g-n}
\end{bmatrix} \\
& \times
\begin{bmatrix}
c_g(a,z) & c_{g-1}(a,z) & \cdots & c_{-g+n}(a,z) & 0 & 0 & \cdots & 0 \\
0 & 0 & \cdots & 0 & s_g(a,z) & s_{g-1}(a,z) & \cdots & s_{-g+n}(a,z)
\end{bmatrix}
v_g(n).
\endaligned
\]
By using the identity 
\[
\begin{bmatrix}
-i & 0 \\
0 & -1 
\end{bmatrix}
\begin{bmatrix}
0 & 1 \\
1 & 0 
\end{bmatrix}
\begin{bmatrix}
m_{2g-n}^{-1} & 0 \\
0 & m_{2g-n} 
\end{bmatrix}
= -
\begin{bmatrix}
0 & -1 \\
1 & 0 
\end{bmatrix}
\begin{bmatrix}
m_{2g-n}^{-1} & 0 \\
0 & m_{2g-n} 
\end{bmatrix}
\begin{bmatrix}
1 & 0 \\
0 & -i 
\end{bmatrix},
\]
we obtain
\[
-a\frac{\partial}{\partial a}
\begin{bmatrix}
A_{q}(a,z) \\
B_{q}(a,z)
\end{bmatrix} 
= z
\begin{bmatrix}
0 & -1 \\
1 & 0 
\end{bmatrix}
\begin{bmatrix}
m_{2g-n}^{-1} & 0 \\
0 & m_{2g-n} 
\end{bmatrix}
\begin{bmatrix}
A_{q}(a,z) \\
B_{q}(a,z)
\end{bmatrix} 
\]
for $q^{(n-1)/2} \leq a < q^{n/2}$. 
This implies (6) by \eqref{def_Hq}. 
\hfill $\Box$  
\medskip 

\noindent 
{\bf Proof of (7).} Note that $E_q(0)=A_q(0)$, since $B_q(z)$ is odd. We have 
\[
\aligned
\lim_{a \to (q^{g})^-}
\begin{bmatrix}
A_{q}(a,z) \\
B_{q}(a,z)
\end{bmatrix}
&= \frac{1}{2}
\begin{bmatrix}
1 & 0 \\
0 & -i 
\end{bmatrix}
\begin{bmatrix}
c_g(q^{g},z) & 0 \\
0 & s_g(q^{g},z) 
\end{bmatrix}
v_g(2g) \\
&= \frac{1}{2}
\begin{bmatrix}
1 & 0 \\
0 & -i 
\end{bmatrix}
\begin{bmatrix}
2 & 0 \\
0 & 0 
\end{bmatrix}
v_g(2g)
=
\begin{bmatrix}
v_g(2g)[1] \\ 0
\end{bmatrix}
\endaligned
\]
by definition \eqref{def_AB}. 
Therefore, it is sufficient to show that 
\[
v_{g}(2g)[1] = 2 \sum_{k=0}^{g-1} c_k + c_g=A_q(0).
\]
In order to prove this, we put 
\[
S_n=P_0^{-1}Q_0 P_1(m_1)^{-1}Q_1 \cdots P_n(m_n)^{-1}Q_n \quad (n=0,1,2,\cdots),
\]
where we understand $m_k$ as parameters. 
The the size of $S_n$ is $2 \times (2n+4)$ by definitions of $P_k(m_k)$ and $Q_k$. 
We have 
\[
S_0=P_0^{-1}Q_0=\begin{bmatrix} 1 & 1 & 0 & 0 \\ 0 & 0 & 1 & -1\end{bmatrix} 
\]
and find that the first row of $S_n$ has the form  
\[
\underbrace{1~1~\cdots~1}_{n+2}~\underbrace{0~0~\cdots~0}_{n+2}
\]
by the induction using the formula of $P_k(m_k)^{-1}Q_k$ in the proof of Lemma \ref{lem_203}. 
By definition \eqref{def_v1}, we have $v_{g}(2g)[1]=(\text{the first row of $S_{2g-1}$})\cdot v_g(0)$. 
Hence we obtain 
\[
v_{g}(2g)[1]=(\underbrace{1~1~\cdots~1}_{2g+1}~\underbrace{0~0~\cdots~0}_{2g+1})\cdot v_g(0)= 2 \sum_{k=0}^{g-1} c_k + c_g = A_q(0).
\]
We complete the proof. \hfill $\Box$  

%
\section{Proof of Theorem \ref{thm_1}} \label{section_6}

%
%
\subsection{Preparations} \label{section_6_1}
%
%
We prepare two propositions for the proof of Theorem \ref{thm_1}. 
\begin{proposition} \label{prop_601}
Let $1 \leq a_1< a_0 \leq g^g$. 
\begin{enumerate}
\item Assume that $m_{q}(a)\not=0$ and $m_{q}(a)^{-1} \not=0$ for every $ 1 \leq  a \leq a_0$. 
Then there exists a $2 \times 2$ matrix valued function $M(a_1,a_0;z)$ 
such that all entries are entire functions of $z$, satisfies 
\begin{equation} \label{prop_601_1}
\begin{bmatrix} 
A_q(a_1,z) \\ B_q(a_1,z)
\end{bmatrix} 
= 
M(a_1,a_0;z)
\begin{bmatrix}
A_q(a_0,z) \\ B_q(a_0,z)
\end{bmatrix},
\end{equation}
and $\det M(a_1,a_0;z)=1$. 
\item Assume that $m_{q}(a)\not=0$ and $m_{q}(a)^{-1} \not=0$ for every $ 1 \leq  a < a_0$. 
Then the matrix valued function $M(a_1,a;z)$ of {\rm (1)} is left-continuous as a function of $a$ and 
\begin{equation} \label{prop_601_2}
\begin{bmatrix} 
A_q(a_1,z) \\ B_q(a_1,z)
\end{bmatrix} 
=
\lim_{a \to a_0^-} M(a_1,a;z)
\lim_{a \to a_0^-}
\begin{bmatrix}
A_q(a,z) \\ B_q(a,z)
\end{bmatrix}
\end{equation}
holds as a vector valued function of $z \in \C$.
\end{enumerate}
\end{proposition}
\begin{proof} 
(1) Put $A(a,z)=A_{q}(a,z)$, $B(a,z)=B_{q}(a,z)$, $m(a)=m_{q}(a)$, and 
\[
J(a) = 
\begin{bmatrix}
0 & - m(a) \\ m(a)^{-1} & 0 
\end{bmatrix}.
\]
Then the system of \eqref{system_1} in Theorem \ref{thm_6} 
is written as 
\begin{equation} \label{system_2}
-a\frac{\partial}{\partial a}
\begin{bmatrix}
A(a,z) \\ B(a,z)
\end{bmatrix}
= z J(a)
\begin{bmatrix}
A(a,z) \\ B(a,z)
\end{bmatrix} \quad (1 \leq a < a_0 ,\,z \in \C).
\end{equation}
By the assumption, $m(a)$ and $m(a)^{-1}$ are integrable on $[a_1, a_0]$. 
Hence, we have 
\begin{equation} \label{601}
\aligned
\, 
&
\begin{bmatrix}
A(a_1,z) \\ B(a_1,z)
\end{bmatrix} 
= 
\begin{bmatrix}
A(a_0,z) \\ B(a_0,z)
\end{bmatrix} 
+
z \int_{a_1}^{a_0}
J(t_1)
\begin{bmatrix}
A(t_1,z) \\ B(t_1,z)
\end{bmatrix} \frac{dt_1}{t_1} \\
& =
\begin{bmatrix}
A(a_0,z) \\ B(a_0,z)
\end{bmatrix} 
+
z \int_{a_1}^{a_0}
J(t_1) \frac{dt_1}{t_1}
\begin{bmatrix}
A(a_0,z) \\ B(a_0,z)
\end{bmatrix} 
+
z^2  \int_{a_1}^{a_0}\int_{t_1}^{a_0}
J(t_1)J(t_2)
\begin{bmatrix}
A(t_2,z) \\ B(t_2,z)
\end{bmatrix} \frac{dt_2}{t_2}\frac{dt_1}{t_1} \\
& =
\left(
I
+
z \int_{a_1}^{a_0}
J(t_1) \frac{dt_1}{t_1}
+
z^2  \int_{a_1}^{a_0}\int_{t_1}^{a_0}
J(t_1)J(t_2)
\frac{dt_2}{t_2}\frac{dt_1}{t_1} \right. \\
&\qquad \qquad \left.
+
z^3  \int_{a_1}^{a_0}\int_{t_1}^{a_0}\int_{t_2}^{a_0}
J(t_1)J(t_2)J(t_3)
\frac{dt_3}{t_3}\frac{dt_2}{t_2}\frac{dt_1}{t_1}
+
\cdots 
\right)
\begin{bmatrix}
A(a_0,z) \\ B(a_0,z)
\end{bmatrix},
\endaligned
\end{equation}
where $I$ is the $2 \times 2$ identity matrix. 
On the other hand, we have 
\[
J(t_1)\cdots J(t_k)
 = (-1)^{k'}
\begin{cases}
\begin{bmatrix}
0 & - \frac{m(t_1)m(t_3)\cdots m(t_{k})}{m(t_2)m(t_4)\cdots m(t_{k-1})} \\ 
\frac{m(t_2)m(t_4)\cdots m(t_{k-1})}{m(t_1)m(t_3)\cdots m(t_{k})} & 0 
\end{bmatrix} & \text{if $k=2k'+1$}, \\[12pt] 
\begin{bmatrix}
\frac{m(t_1)m(t_3)\cdots m(t_{k-1})}{m(t_2)m(t_4)\cdots m(t_{k})} & 0 \\ 
0 & \frac{m(t_2)m(t_4)\cdots m(t_{k})}{m(t_1)m(t_3)\cdots m(t_{k-1})}
\end{bmatrix} & \text{if $k=2k'$}.
\end{cases}
\]
Therefore, by taking 
\[
C(a_0,a_1):=\sup \{m(a),m(a)^{-1};~ a \in [a_1,a_0]\}
\]
and by using the formula 
\[
\int_{a_1}^{a_0}\int_{t_1}^{a_0}\int_{t_2}^{a_0} \cdots \int_{t_{k-1}}^{a_0} 
1 \,\frac{dt_k}{t_k} \cdots \frac{dt_2}{t_2} \frac{dt_1}{t_1} = \frac{1}{k!}\left(\log \frac{a_0}{a_1}\right)^k, 
\]
we obtain 
\[
\left|
\left[\int_{a_1}^{a_0}\int_{t_1}^{a_0}\int_{t_2}^{a_0} \cdots \int_{t_{k-1}}^{a_0} 
J(t_1)\cdots J(t_k) \, \frac{dt_k}{t_k} \cdots \frac{dt_2}{t_2} \frac{dt_1}{t_1} \right]_{ij} \right| \leq  
\frac{1}{k!}C(a_0,a_1)^k\left(\log \frac{a_0}{a_1}\right)^k,
\]
for every $1 \leq i,j \leq 2$, where $[M]_{ij}$ means the $(i,j)$-entry of a matrix $M$. 
This estimate implies that 
the right-hand side of \eqref{601} 
converges absolutely and uniformly 
if $z$ lie in a bounded region. 

Suppose that $m(a)=m\not=0$ for $a_1 \leq a \leq a_0$. 
Then 
\[
I
+
z \int_{a_1}^{a_0}
J(t_1) 
\frac{dt_1}{t_1}
+
z^2  \int_{a_1}^{a_0}\int_{t_1}^{a_0}
J(t_1)J(t_2)
\frac{dt_2}{t_2} \frac{dt_1}{t_1}
+
z^3  \int_{a_1}^{a_0}\int_{t_1}^{a_0}\int_{t_2}^{a_0}
J(t_1)J(t_2)J(t_3)
\frac{dt_3}{t_3} \frac{dt_2}{t_2} \frac{dt_1}{t_1} 
+
\cdots 
\]
is equal to 
\[
\begin{bmatrix}
\cos(z\log(a_0/a_1)) & - m \sin(z\log(a_0/a_1)) \\ \frac{1}{m} \sin(z\log(a_0/a_1)) & \cos(z\log(a_0/a_1))
\end{bmatrix}
\]
and hence \eqref{prop_601_1} holds by taking this matrix as $M(a_1,a_0;z)$. 
Therefore, 
if we suppose that $m(a)=m_j\not=0$ on $[t_{j+1},t_j)$ 
for a partition $[a_1,a_0]=[a_1,t_{k-1}) \cup \cdots \cup [t_1,a_0]$ 
with $t_0=a_0$ and $t_k=a_1$,   
then we have \eqref{prop_601_1} by taking 
\[
\aligned
M&(a_0,a_1;z) \\
:=&
\begin{bmatrix}
\cos(z\log(t_{k-1}/a_1)) & - m_k \sin(z\log(t_{k-1}/a_1)) \\ \frac{1}{m_k} \sin(z\log(t_{k-1}/a_1)) & \cos(z\log(t_{k-1}/a_1))
\end{bmatrix} \\
& \times 
\begin{bmatrix}
\cos(z\log(t_{k-2}/t_{k-1})) & - m_{k-1} \sin(z\log(t_{k-2}/t_{k-1})) \\ \frac{1}{m_{k-1}} \sin(z\log(t_{k-2}/t_{k-1})) & \cos(z\log(t_{k-2}/t_{k-1}))
\end{bmatrix}
\times \cdots \\
& \times \begin{bmatrix}
\cos(z\log(t_1/t_2)) & - m_2 \sin(z\log(t_1/t_2)) \\ \frac{1}{m_2} \sin(z\log(t_1/t_2)) & \cos(z\log(t_1/t_2))
\end{bmatrix}
\begin{bmatrix}
\cos(z\log(a_0/t_1)) & - m_1 \sin(z\log(a_0/t_1)) \\ \frac{1}{m_1} \sin(z\log(a_0/t_1)) & \cos(z\log(a_0/t_1))
\end{bmatrix}.
\endaligned
\]
Moreover, $\det M(a_1,a_0;z)=1$ is obvious by this definition. 
Now we complete the proof, since $m(a)$ is a constant on $[q^{(n-1)/2},q^{n/2})$ for every $1 \leq n \leq 2g$ by definition \eqref{def_mq}. 
\smallskip

\noindent
(2) By the above definition, $M(a_1,a;z)$ is left-continuous with respect to $a$, 
since $m_q(a)$ is left-continuous by definition \eqref{def_mq}. 
Because $A_q(a,z)$ and $B_q(a,z)$ are left-continuous 
with respect to $a$ by Theorem \ref{thm_6} (4), 
we obtain \eqref{prop_601_2} from \eqref{prop_601_1}. 
\end{proof}

\begin{corollary} \label{cor_602}
Let $m_q(a)$ be of \eqref{def_mq}. 
Assume that $m_{q}(a)\not=0$ and $m_{q}(a)^{-1} \not=0$ for every $ 1 \leq  a < q^g$. 
Then we have
\[
\aligned
\begin{bmatrix}
A_q(a,z) \\ B_q(a,z)
\end{bmatrix}
& = E_q(0)
\begin{bmatrix}
\cos(z\log(q^{n/2}/a)) & -m_{2g-n}\sin(z\log(q^{n/2}/a))\\
m_{2g-n}^{-1}\sin(z\log(q^{n/2}/a)) & \cos(z\log(q^{n/2}/a))
\end{bmatrix} \\
& \quad \times
\prod_{k=1}^{2g-n}
\begin{bmatrix}
\cos(\frac{z}{2}\log q) & -m_{2g-(n+k)}\sin(\frac{z}{2}\log q)\\
m_{2g-(n+k)}^{-1}\sin(\frac{z}{2}\log q) & \cos(\frac{z}{2}\log q)
\end{bmatrix}
\begin{bmatrix}
1 \\ 0
\end{bmatrix}
\endaligned
\]
for $q^{(n-1)/2} \leq a < q^{n/2}$. 
\end{corollary}
\begin{proof}
By Theorem \ref{thm_6}, we can apply the proof of Proposition \ref{prop_601} (1) to 
$a_1=a$ ($q^{(n-1)/2} \leq a < q^{n/2}$), $a_0=q^g$ and $t_k=q^{(g-k)/2}$ ($1 \leq k \leq g-n$) 
and then we obtain the formula.
\end{proof}

\begin{proposition} \label{lem_602}
Define 
\begin{equation} \label{lem_602_1}
K(a;z,w)
:=\frac{\overline{E_q(a,w)}E_q(a,z)-\overline{E_q^\sharp(a,w)}E_q^\sharp(a,z)}{2\pi i(\bar{w}-z)}. 
\end{equation}
Then we have 
\begin{equation} \label{lem_602_2}
K(a;z,w)
=\frac{\overline{A_q(a,w)}B_q(a,z)-\overline{B_q(a,w)}A_q(a,z)}{\pi(z-\bar{w})}.
\end{equation}
Moreover, if $m_q(a)$ and $m_q(a)^{-1}$ are integrable on $[a_1,a_0]$, then we have 
\begin{equation} \label{lem_602_3}
\aligned
K&(a_1;z,w) - K(a_0;z,w) \\
& \qquad = \frac{1}{\pi}\int_{a_1}^{a_0}\overline{A_q(a,w)}A_q(a,z) \, \frac{1}{m_q(a)} \, \frac{da}{a} 
+ \frac{1}{\pi}\int_{a_1}^{a_0}\overline{B_q(a,w)}B_q(a,z) \, m_q(a) \, \frac{da}{a} 
\endaligned
\end{equation}
for every $z,w \in \C$. 
\end{proposition}
\begin{proof}
We obtain \eqref{lem_602_2} easily  
by substituting \eqref{def_E} and \eqref{def_E_2} into \eqref{lem_602_1}. 
By the integration by parts together with \eqref{system_1}, we obtain 
\[
\aligned
z \int_{a_1}^{a_0}\overline{A(a,w)}A(a,z) \, \frac{1}{m(a)} \, \frac{da}{a}
& = - \left.\overline{A(a,w)}B(a,z)\right|_{a_1}^{a_0} 
+ \bar{w} \int_{a_1}^{a_0}\overline{B(a,w)} B(a,z) \, m(a) \, \frac{da}{a},
\endaligned
\]
\[
\aligned
z \int_{a_1}^{a_0}\overline{B(a,w)}B(a,z) \, m(a) \, \frac{da}{a} 
& = \left.\overline{B(a,w)}A(a,z)\right|_{a_1}^{a_0} 
+ \bar{w} \int_{a_1}^{a_0}\overline{A(a,w)} A(a,z) \, \frac{1}{m(a)} \, \frac{da}{a} .
\endaligned
\]
Moving the second terms of the right-hand sides of the two equations to the left-hand sides, 
then adding both sides of the resulting two equations, 
and finally dividing both sides by $(z-\bar{w})$, 
\[
\aligned
\int_{a_1}^{a_0} & \overline{A(a,w)}A(a,z) \, \frac{1}{m(a)} \, \frac{da}{a}  
+ \int_{a_1}^{a_0}\overline{B(a,w)}B(a,z) \, m(a) \, \frac{da}{a}  \\
&= \frac{\left.(-\overline{A(a,w)}B(a,z)+\overline{B(a,w)}A(a,z))\right|_{a_1}^{a_0}}{z-\bar{w}} 
 = \pi\Bigl( K(a_1;z,w) - K(a_0;z,w) \Bigr).
\endaligned
\]
This implies \eqref{lem_602_3}. 
\end{proof}

%
%
\subsection{Proof of necessity } \label{section_6_2}
%
%

By Lemma \ref{lem_401}, we see that 
all zeros of $P_g(x)$ lie on $T$ and simple
if and only if 
$E_q(z)$ is a function of class HB  . 
Thus 
$P_g(x)$ has a multiple zero or a zero outside $T$
if and only if 
$E_q(z)$ is not a function of class HB. 
Hence, for the necessity, it is sufficient  to prove that 
$E_q(z)$ is not a function of class HB 
if $m_{2g-n}(\underline{c}\,;\log q) \leq 0$ 
or $m_{2g-n}(\underline{c}\,;\log q)^{-1} \leq 0$ 
for some $1 \leq n \leq 2g$. 
We will prove it in three steps as follows, 
there we put $m_{2g-n}=m_{2g-n}(\underline{c}\,;\log q)$. 
\medskip

{\bf Step 1.} Firstly, we note that $m_{2g-1}>0$ by \eqref{m_2g-1}. 
Therefore, if $m_{2g-n}(\underline{c}\,;\log q) \leq 0$ 
or $m_{2g-n}(\underline{c}\,;\log q)^{-1} \leq 0$ 
for some $1 \leq n \leq 2g$, 
then there exists  $1 \leq n_0 \leq 2g-1$ 
such that $m_{2g-n}>0$ and $m_{2g-n}^{-1}>0$ for every $1 \leq n \leq n_0$, 
$m_{2g-(n_0+1)}(\underline{c}\,;\log q) \leq 0$ 
or $m_{2g-(n_0+1)}(\underline{c}\,;\log q)^{-1} \leq 0$,  
$A_q(a,z)$ and $B_q(a,z)$ are defined for $1 \leq a < q^{n_0/2}$ and 
\begin{equation} \label{603}
\begin{bmatrix}
A_q(z) \\ B_q(z)
\end{bmatrix} 
=
\begin{bmatrix}
A_q(1,z) \\ B_q(1,z)
\end{bmatrix} 
= 
M(1,a;z)
\begin{bmatrix}
A_q(a,z) \\ B_q(a,z)
\end{bmatrix} 
\end{equation}
holds for $1 \leq a < q^{n_0/2}$ by applying \eqref{prop_601_1} to $(a_1,a_0)=(1,a)$. 

{\bf Step 2.} Let $n_0$ be the number of Step 1. 
Suppose that $E_q(a,z)=A_q(a,z)-iB_q(a,z)$ is not a function of HB for some $1 < a_0 \leq  q^{n_0/2}$, 
that is, $E_q(a_0,z)$ has a real zero for some $1<a_0 \leq q^{n_0/2}$ 
or $|E_q^\sharp(a_0,z)| \geq |E_q(a_0,z)|$ for some ${\rm Im}\, z>0$ and $1<a_0 \leq q^{n_0/2}$. 
Here, we understand $A_q(a_0,z)$ and $B_q(a_0,z)$ 
in the sense of left-sided limit $a \to a_0^-$ if $a_0=q^{n_0/2}$ (see \eqref{prop_601_2}). 

If $E_q(a,z)$ has a real zero for some $1<a_0 \leq q^{n_0/2}$, 
then $A_q(a_0,z)$ and $B_q(a_0,z)$ have a common real zero, 
since they are real valued on the real line.  
Therefore, \eqref{603} and $\det M(1,a_0;z)=1$ imply that $A_q(z)$ and $B_q(z)$ have a common real zero. 
Hence $E_q(z)$ has a real zero and thus $E_q(z)$ is not a function of class HB. 

On the other hand, we assume that $E_q(a,z)$ has no real zeros for every $1 < a \leq q^{n_0/2}$ 
but it has a zero in the upper half plane for some $1 < a_0 \leq  q^{n_0/2}$. 
By \eqref{def_AB} and Theorem \ref{thm_6} (3), $E_q(a,z)$ is a continuous function of $(a,z)\in [1,q^{n_0/2}]\times \C$. 
Therefore, any zero locus of $E_q(a,z)$ is a continuous curve in $\C$ parametrized by $a\in[1,q^{n_0/2}]$. 
Denote by $z_a \subset \C$ a zero locus through a zero of $E_q(a_0,z)$ in the upper-half plane, 
that is, $E_q(a,z_a)=0$ for every $1 \leq a \leq q^{n_0/2}$. 
If ${\rm Im}(z_{a_1})<0$ for some $1\leq a_1 < a_0$, 
then ${\rm Im}(z_{a_2})=0$ for some $a_1<a_2<a_0$. 
This implies that $E_q(a_2,z)$ has a real zero at $z=z_{a_2}$. This is a contradiction. 
Therefore, ${\rm Im}(z_a) \geq 0$ for every $1 \leq a < a_0$, 
in particular ${\rm Im}(z_1) \geq 0$. 
This implies $E_q(z)=E_q(1,z)$ is not a function of class HB. 

If $E_q(a,z)\not=0$ for ${\rm Im}\, z \geq 0$ 
but $|E_q^\sharp(a_0,z_0)| \geq |E_q(a_0,z_0)|$ for some $1 < a_0 \leq q^{n_0/2}$ and ${\rm Im}(z_0)>0$, 
then it derives a contradiction. 
Because $A_q(a,z)$ and $B_q(a,z)$ are bounded on the real line by definition \eqref{def_AB}, 
$E_q(a,z)$ is a function of the Cartwright class~\cite[the first page of Chapter II]{Levin96}. 
Therefore, we have the factorization
\[
E_q(a_0,z) = C \lim_{R \to \infty} \prod_{|\rho|<R}\left(1-\frac{z}{\rho} \right) 
\]
(see \cite[Remark 2 of Lecture 17.2]{Levin96}). Here ${\rm Im}(\rho)<0$ for every zero of $E_q(a_0,z)$ by the assumption. 
Hence, we have
\[
\left|\frac{E_q^\sharp(a_0,z)}{E_q(a_0,z)}\right| 
= \lim_{R \to \infty} \prod_{|\rho|<R}\left| \frac{z-\bar{\rho}}{z-\rho} \right|<1 \quad \text{for} \quad {\rm Im}\, z>0.
\]
This contradict to the assumption $|E_q^\sharp(a_0,z_0)| \geq |E_q(a_0,z_0)|$. 

{\bf Step 3.} For the number $n_0$ of Step 1, 
one of the following case occurs: 
\begin{enumerate}
\item[(i)] $m_{2g-(n_0+1)} = 0$ or $m_{2g-(n_0+1)}^{-1}=0$,  
\item[(ii)] $m_{2g-(n_0+1)}<0$ and $m_{2g-(n_0+1)}^{-1} \not=0$. 
\end{enumerate}
We prove that $E_q(z)$ is not a function of HB even if whichever occurs. 
Considering the argument in Step 2, 
we assume that $E_q(a,z)$ is a function of class HB for every $1<a\leq q^{n_0/2}$ 
in both cases. 
\smallskip

{\bf Case (i).} 
Suppose that $m_{2g-(n_0+1)}=0$. Then, by definition \eqref{def_m2}, we have 
\[
v_g(n_0)[1]+v_g(n_0)[2g-n_0+1]=0. 
\] 
This implies that the function 
\[
\aligned
\phi(z)=\lim_{a \to (q^{n_0/2})^-} A_q(a,z) 
= \frac{1}{2}(T_{n_0}(q^{n_0/2},z) v_g(n_0))[1] 
\quad ((q^{1/2})^{2g-n_0-2})
\endaligned
\] 
is of exponential type whose mean type is at most $(g-1-(n_0/2))\log q$. 
If 
\[
v_g(n_0)[2g-n_0+2] +v_g(n_0)[4g-2n_0+2] =0 
\] 
for the denominator of \eqref{def_m2}, 
then the function 
\[
\psi(z)=\aligned
\lim_{a \to (q^{n_0/2})^-} B_q(a,z) 
= -\frac{i}{2}(T_{n_0}(q^{n_0/2},z) v_g(n_0))[2] 
\endaligned
\]
is also of exponential type whose mean type is at most $(g-1-(n_0/2))\log q$.
On the other hand, 
we have 
\[
\begin{bmatrix}
A_q(1,z) \\ B_q(1,z)
\end{bmatrix} 
= 
\lim_{a \to (q^{n_0/2})^-}
M(1,a;z)
\begin{bmatrix}
\phi(z) \\ \psi(z)
\end{bmatrix}
\]
by \eqref{prop_601_2}. 
Here, the entries of the left-hand side are entire functions of exponential type 
with mean type $g\log q$, 
while entries of the right-hand side are entire functions of exponential type 
with mean type at most $(g-1)\log q$ from the construction of $M(1,a;z)$ in the proof of Proposition \ref{prop_601} (1). 
This is a contradiction. Hence, it must be
\[
v_g(n_0)[2g-n_0+2] +v_g(n_0)[4g-2n_0+2] \not=0
\] 
and $\psi(z)$ is a function of exponential type whose mean type is just $(g-(n_0/2))\log q$. 

By the assumption, $E_q(a,z)$ is a function of HB at $a=q^{n_0/2}$ (in the sense of left-sided limit). 
Therefore, $\phi(z)$ and $\psi(z)$ have only real zeros 
and their zeros interlace. 
However, by \cite[Theorem 1 of Lecture 17.2]{Levin96}, 
main terms of asymptotic formulas for the number of (real) zeros of $\phi(z)$ and $\psi(z)$ 
in $[-T,T] \subset \R$ are strictly different. 
In particular, their (real) zeros can not interlace. 
This is a contradiction. 
Hence $E_q(a,z)$ is not a function of HB for some $1<a \leq q^{n_0/2}$ 
which implies that $E_q(z)$ is not a function of HB by Step 2. 
The case of $m_{2g-n_0}^{-1}=0$ is proved by a similar way. 
\smallskip

{\bf Case (ii).} 
In this case, we can assume that $E_q(a,z)$ is a function of class HB for every $1< a \leq q^{(n_0+1)/2}$, 
since we only used $m_{2g-n}\not=0$ and $m_{2g-n}^{-1}\not=0$ in Step 2. 
Put $a_1=q^{n_0/2}$, $a_0=(q^{(n_0+1)/2}-q^{n_0/2})/2$ and $m_{2g-(n_0+1)}=-m<0$. 
Then, for every $a_1 \leq a \leq a_0$, 
we have $m_q(a)=-m$ by \eqref{def_mq} and find that  
$E_q(a,z)$ generates the de Branges space $B(E_q(a,z))$ 
which is the Hilbert space of all entire functions $F(z)$ such that 
$\int_{\R}|F(x)/E_q(a,x)|^2 dx<\infty$ 
and $F(z)/E_q(a,z)$, $F(z)/E_q^\sharp(a,z)$ are 
functions of the Hardy space $H^2$ in the upper half-plane 
(see \cite[\S19]{deBranges68} and \cite[Proposition 2.1]{Remling02}).

We have $K(a_0;z,z) > K(a_1;z,z)$ by applying \eqref{lem_602_3} to $z=w$ with $m_q(a)=-m<0$. 
Therefore, it follows that for every $f \in B(E_q(a_1,z))$
\[
|f(z)|^2 \leq \Vert f \Vert_{a_1}^2K(a_1;z,z) < \Vert f \Vert_{a_1}^2K(a_0;z,z)
\]
by \cite[Theorem 20]{deBranges68}, 
where $\Vert \cdot \Vert_{a_1}$ is the norm of $B(E_q(a_1,z))$. 
By applying this to the function 
\[
g(z):=\frac{E_q(a_1,z)-E_q(a_1,iy_0)}{z-iy_0} \quad (y_0 \in \R)
\] 
which is a function of $B(E_q(a_1,z))$ 
by Lemma 3.3 and Lemma 3.4 of \cite{Dym70}, 
we obtain  
\[
\aligned
|g(iy)|^2
\leq \Vert g \Vert_{a_1}^2 K(t_0,iy,iy) =\Vert g \Vert_{a_1}^2
\frac{|E_q(a_0,iy)|^2-|E_q^\sharp(a_0,z)|^2}{4\pi y}
\leq \Vert g \Vert_{a_1}^2
\frac{|E_q(a_0,iy)|^2}{4\pi y}. 
\endaligned
\]
By $E_q(a,z)=A_q(a,z)-iB_q(a,z)$ with \eqref{def_AB}, we see that 
\[
y^{-1}q^{(g-\frac{n_0-1}{2})y} \ll |g(iy)| \ll 
y^{-1/2}|E_q(a_0,iy)| \ll y^{-1/2} q^{(g-\frac{n_0}{2})y} 
\quad \text{as $y \to +\infty$}. 
\]
This is a contradiction. 
Hence $E_q(a,z)$ is not a function of HB for some $1<a \leq a_0$ 
and it implies that $E_q(a,z)$ is not a function of HB by Step 2. 
\hfill $\Box$
%
%
\subsection{Proof of sufficiency} \label{section_6_3}
%
%
Suppose that  $m_{2g-n}(\underline{c}\,;\log q) > 0$ 
and $m_{2g-n}(\underline{c}\,;\log q)^{-1} > 0$
for every $1 \leq n \leq 2g$. 
Then $m_q(a)$ and $m_q(a)^{-1}$ are both integrable on $[1,q^g)$ and positive real valued. 
Therefore, by applying \eqref{lem_602_3} to $(a_1,a_0)=(a,b)$ 
and then by tending $b$ to $q^g$ together with Theorem \ref{thm_6} (7) and \eqref{lem_602_1}, we have  
\[
\aligned
0 & 
< \frac{1}{\pi}\int_{a}^{q^g} |A_q(t,z)|^2 \, \frac{1}{m_q(t)} \, \frac{dt}{t} 
+ \frac{1}{\pi}\int_{a}^{q^g} |B_q(t,z)|^2 \, m_q(t) \, \frac{dt}{t} \\
& \quad = K(a;z,z) - \lim_{b \to q^g} K(b;z,z) = \frac{|E_q(a,z)|^2-|E_q^\sharp(a,z)|^2}{4\pi {\rm Im}\, z}
\endaligned
\]
for every $1 \leq a < q^g$ if ${\rm Im}\, z>0$. 
Thus $E_q(a,z)$ is a function of class $\overline{\rm HB}$ for every $1 \leq a < q^g$. 
In addition, we have 
\[
\begin{bmatrix}
A_q(a,z) \\ B_q(a,z)
\end{bmatrix} 
= \left( \lim_{b \to (q^g)^-}
M(a,b;z) \right)
\begin{bmatrix}
E_q(0) \\ 0
\end{bmatrix} \quad \text{with} \quad \lim_{b \to (q^g)^-}\det M(a,b;z)=1
\]
for every $1 \leq a < q^g$. 
Here $E_q(0)=A_q(0)$ and   
\[
0 \not= m_0(\underline{c}\,;\log q)=\frac{v_g(2g)[1]}{v_g(2g)[2]}=\frac{A_q(0)}{v_g(2g)[2]}
\] 
by the proof of Theorem \ref{thm_6} (7). If $A_q(0)=v_g(2g)[1]=v_g(2g)[2]=0$, 
then $A_q(a,z)$ and $B_q(a,z)$ are both identically zero for $q^{(2g-1)/2} \leq a < q^g$ by \eqref{def_AB}. 
It implies that $E_q(z)=A_q(z)=B_q(z) \equiv 0$ by Proposition \ref{prop_601} together with the assumption. 
Therefore, $A_q(0)\not=0$ and it implies that $A_q(a,z)$ and $B_q(a,z)$ have no common zeros 
for every $1 \leq a <q^g$. 
Thus $E_q(a,z)$ has no real zeros. 
As a consequence $E_q(a,z)$ is a function of class $\rm HB$ for every $1 \leq a < q^g$. 
Hence $A_q(z)$ has only simple real zeros by $E_q(z)=E_q(1,z)$ and Lemma \ref{lem_401} (2). 
\hfill $\Box$
%
%
\section{Proof of Theorem \ref{thm_3}, \ref{thm_5}, \ref{thm_7} and Corollary \ref{cor_2}} \label{section_7}
%
%

Firstly, we prove Theorem \ref{thm_3} 
by using Theorem \ref{thm_7} as well as the proof of Theorem \ref{thm_1}. 
However, we omit the proof of Theorem \ref{thm_7}, 
because it is proved by a way similar to the proof of 
Theorem \ref{thm_6} in Section \ref{section_5}. 
Also, we omit the proof of Theorem \ref{thm_4}, because it is just a rewriting of Theorem \ref{thm_3} 
by \eqref{def_m3_2} and \eqref{def_R2}.
Successively, we prove Corollary \ref{cor_2} by using Theorem \ref{thm_3} and Theorem \ref{thm_7} 
as well as Corollary \ref{cor_1}. 
Finally, we prove Theorem \ref{thm_7}. 
%
%
\subsection{Proof of Theorem \ref{thm_3} }
%
%
We start from the following lemma. 

\begin{lemma} \label{lem_2_1}
Let $E_{q,\omega}(z)$ be in \eqref{def_E2}. 
Then, all zeros of $P_g(x)$ lie on $T$ 
if and only if 
$E_{q,\omega}(z)$ is a function of class ${\rm HB}$ 
for every $\omega>0$. 
\end{lemma}
\begin{proof}
By definition \eqref{def_A}, 
all zeros of $P_g(x)$ lie on $T$ 
if and only if $A_q(z)$ has only real zeros. 
Suppose that $A_q(z)$ has only real zeros (allowing multiple zeros). 
Then $E_{q,\omega}(z)$ satisfies inequality \eqref{HB} for every $\omega>0$, 
because we have
\[
\aligned
\left| \frac{\overline{E_{q,\omega}(\bar{z})}}{E_{q,\omega}(z)} \right|^2 
& = 
\left| \frac{A_q(z-i\omega)}{A_q(z+i\omega)} \right|^2 = 
\prod_{\rho} \left| \frac{(x-\rho)+i(y-\omega)}{(x-\rho)+i(y+\omega)} \right|^2 \\
& = 
\prod_{\rho} \left( 1 - \frac{4 \omega y}{(x - \rho)^2 + (y+\omega)^2} \right) <1
\endaligned
\]
for $z=x+iy$ with $y>0$ by using the factorization \eqref{H-factorization}. 
Moreover, $E_{q,\omega}(z)$ has no real zeros for every $\omega>0$ 
by definition \eqref{def_E2} and the assumption. 
Hence $E_{q,\omega}(z)$ is a function of class ${\rm HB}$ for every $\omega>0$. 

Conversely, suppose that $E_{q,\omega}(z)$ is a function of class ${\rm HB}$ for every $\omega>0$. 
Then all zeros of $A_{q,\omega}(z)$ and $B_{q,\omega}(z)$ are real, simple, 
and they interlace by Proposition \ref{lem_303}. 
In particular, $A_{q,\omega}(z)$ has only real zeros for every $\omega>0$. 
Hence $A_q(z)=\lim_{\omega \to 0^+}A_{q,\omega}(z)$ has only real zeros 
by Hurwitz's theorem in complex analysis (\cite[Th. (1,5)]{Marden66}). 
\end{proof}

\noindent
{\bf Proof of necessity.} 
By Lemma \ref{lem_2_1}, $P_g(x)$ has a zero outside $T$
if and only if 
$E_{q,\omega}(z)$ is not a function of class HB 
for some $\omega>0$.  
Hence it is sufficient to prove that 
if there exists $\omega_0>0$ 
such that  $m_{2g-n}(\underline{c}\,; q^{\omega_0}) \leq 0$ 
or $m_{2g-n}(\underline{c}\,; q^{\omega_0})^{-1} \leq 0$ 
for some $1 \leq n \leq 2g$, 
then $E_{q,\omega_0}(z)$ is not a function of class HB. 
This is proved by a way similar to the proof of Section \ref{section_6_2} 
by using Theorem \ref{thm_7} instead of Theorem \ref{thm_6}.  
\hfill $\Box$
\medskip

\noindent
{\bf Proof of sufficiency.} 
Let $\omega>0$. Suppose that $m_{2g-n}(\underline{c}\,;q^\omega) > 0$ 
and $m_{2g-n}(\underline{c}\,;q^\omega)^{-1} > 0$
for every $1 \leq n \leq 2g$.  
Then it is proved that $E_{q,\omega}(z)$ is a function of class HB 
by a way similar to the proof of Section \ref{section_6_3}  
by using Theorem \ref{thm_7} instead of Theorem \ref{thm_6}. 
Therefore, by Lemma \ref{lem_2_1}, 
if $m_{2g-n}(\underline{c}\,;q^\omega) > 0$ 
and $m_{2g-n}(\underline{c}\,;q^\omega)^{-1} > 0$
for every $1 \leq n \leq 2g$ and $\omega>0$, 
then all zeros of $P_g(x)$ lie on $T$.     
\hfill $\Box$
%
%
\subsection{Proof of Corollary \ref{cor_2}}
%
%
If all zeros of $P_g(x)$ lie on $T$, 
then $A_{q,\omega}(z)$ has only simple real zeros 
for every $\omega>0$ by Proposition \ref{lem_303} and Lemma \ref{lem_2_1}. 
In particular, $E_{q,\omega}(0)=A_{q,\omega}(0)\not=0$, since $A_{q,\omega}(z)$ is an even 
function having only simple real zeros. 
Hence,  
$H_{q,\omega}(a)$ defines a canonical system on $[1,q^g)$ 
having the solution $(A_{q,\omega}(a,z)/E_{q,\omega}(0)$, $B_{q,\omega}(a,z)/E_{q,\omega}(0))$ 
by Theorem \ref{thm_3} and Theorem \ref{thm_7}. 
Conversely, if $H_{q,\omega}(a)$ defines a canonical system for every $\omega>0$, 
then $m_{2g-n}(\underline{c}\,;q^\omega)>0$ 
and $m_{2g-n}(\underline{c}\,;q^\omega)^{-1}>0$ for every $1 \leq n \leq 2g$ and $\omega>0$
by Definition \ref{def_301}.   
Hence, all zeros of $P_g(x)$ lie on $T$ by Theorem \ref{thm_3}. 
%
%
\subsection{Proof of Theorem \ref{thm_5}} \label{section_7_3}
%
%
We define $\tilde{m}_{2g-n}(\underline{c}\,;q^\omega)$ for $1 \leq n \leq 2g$
by \eqref{def_m1} and \eqref{def_v1} starting from the initial vector 
$\displaystyle{\tilde{v}_g(0) =
\begin{bmatrix}
{\mathbf a}_{g,\omega}(0) \\
\omega^{-1}{\mathbf a}_{g,\omega}(0)
\end{bmatrix}}$, 
where ${\mathbf a}_{g,\omega}(0)$ is the vector of \eqref{def_v0_2}. 
In addition, we take 
$\displaystyle{\tilde{m}_{2g}(\underline{c}\,;q^\omega)
:=\omega \frac{q^{g\omega}+q^{-g\omega}}{q^{g\omega}-q^{-g\omega}}}$.  
Then we have
\[
\tilde{m}_{2g-n}(\underline{c}\,;q^\omega)=\omega\cdot m_{2g-n}(\underline{c}\,;q^\omega),
\]
for every $1 \leq n \leq 2g$, 
where $m_{2g-n}(\underline{c}\,;q^\omega)$ is of \eqref{def_m3}. 
This implies  
\[
R_n(\underline{c}\,;q^\omega)=\tilde{R}_n(\underline{c}\,;q^\omega)
\]
for every $1 \leq n \leq 2g$  
if we define $\tilde{R}_n(\underline{c}\,;q^\omega)$ by 
\[
\tilde{R}_n(\underline{c}\,;q^\omega) = 
\begin{cases}
\displaystyle{\prod_{j=0}^{J} \frac{\tilde{m}_{2g-(2j+1)}(\underline{c}\,;q^\omega)}{\tilde{m}_{2g-2j}(\underline{c}\,;q^\omega)}} & \text{if $n=2J+1 \geq 1$}, \\
\displaystyle{\prod_{j=0}^{J} \frac{\tilde{m}_{2g-(2j+2)}(\underline{c}\,;q^\omega)}{\tilde{m}_{2g-(2j+1)}(\underline{c}\,;q^\omega)}} & \text{if $n=2J+2 \geq 2$}.
\end{cases}
\]
Therefore, we obtain
\[
\aligned
m_{2g-n}(\underline{c}\,;q^\omega)
& = \frac{q^{g\omega}+q^{-g\omega}}{q^{g\omega}-q^{-g\omega}}\,\tilde{R}_{n-1}(\underline{c}\,;q^\omega)\tilde{R}_n(\underline{c}\,;q^\omega) \\
& = \frac{1}{g} \left(\frac{1}{\omega \log q}+O(\log q^\omega)\right)\tilde{R}_{n-1}(\underline{c}\,;q^\omega)\tilde{R}_n(\underline{c}\,;q^\omega) 
\quad (q^\omega \to 1^+).
\endaligned
\]
Hence, for Theorem \ref{thm_5}, it is sufficient to prove that 
\[
\lim_{q^\omega \to 1^+}\tilde{R}_n(\underline{c}\,;q^\omega)=R_n(\underline{c}) \quad (1 \leq n \leq 2g). 
\]
By the definitions of $\tilde{m}_{2g-n}(\underline{c}\,;q^\omega)$ and $m_{2g-n}(\underline{c}\,;\log q)$, 
this equality follows from the following formula of $v_g(1)$: 
\[
\lim_{q^\omega \to 1^+}
P_{2g-1}(m_{2g-1})^{-1}Q_{2g-1}
\begin{bmatrix}
{\mathbf a}_{g,\omega}(0) \\
\omega^{-1}{\mathbf a}_{g,\omega}(0)
\end{bmatrix}
=
P_{2g-1}(m_{2g-1})^{-1}Q_{2g-1}
\begin{bmatrix}
{\mathbf a}_{g}(0) \\
{\mathbf b}_{g}(0)
\end{bmatrix}
=
\begin{bmatrix}
{\mathbf a}_g(1) \\
{\mathbf b}_g(1)
\end{bmatrix},
\]
where the right-hand side is the vector of \eqref{vec_v1}. 
Put
\[
\begin{bmatrix}
\tilde{{\mathbf a}}_{g,\omega}(1) \\
\tilde{{\mathbf b}}_{g,\omega}(1)
\end{bmatrix}
:=
P_{2g-1}(\tilde{m}_{2g-n}(\underline{c}\,;q^\omega))^{-1}Q_{2g-1}
\begin{bmatrix}
{\mathbf a}_{g,\omega}(0) \\
\omega^{-1}{\mathbf a}_{g,\omega}(0)
\end{bmatrix}.
\]
Then, by using the formula of $P_{k}(m_{k})^{-1}Q_{k}$ in the proof of Lemma \ref{lem_203}, 
we have 
\[
\aligned
\tilde{a}_{g,\omega}(1)
& =\begin{bmatrix}
2\cosh(g\log q^\omega)c_0 \\
(\cosh((g-1)\log q^\omega)+\omega^{-1} m_{2g-1}\sinh((g-1)\log q^\omega))c_1 \\
(\cosh((g-2)\log q^\omega)+\omega^{-1} m_{2g-1}\sinh((g-2)\log q^\omega))c_2 \\
\vdots \\
(\cosh(\log q^\omega)+\omega^{-1} m_{2g-1}\sinh(\log q^\omega))c_{g-1} \\
c_g \\
(\cosh(\log q^\omega)-\omega^{-1} m_{2g-1}\sinh(\log q^\omega))c_{g-1} \\
\vdots \\
(\cosh((g-2)\log q^\omega)-\omega^{-1} m_{2g-1}\sinh((g-2)\log q^\omega))c_2 \\
(\cosh((g-1)\log q^\omega)-\omega^{-1} m_{2g-1}\sinh((g-1)\log q^\omega))c_1 
\end{bmatrix} \\
\tilde{b}_{g,\omega}(1)
& =
\begin{bmatrix}
2 \, \omega^{-1}\sinh(g\log q^\omega)c_0 \\
m_{2g-1}^{-1}(\cosh((g-1)\log q^\omega)+\omega^{-1} m_{2g-1}\sinh((g-1)\log q^\omega))c_1 \\
m_{2g-1}^{-1}(\cosh((g-2)\log q^\omega)+\omega^{-1} m_{2g-1}\sinh((g-2)\log q^\omega))c_2 \\
\vdots \\
m_{2g-1}^{-1}(\cosh(\log q^\omega)+\omega^{-1} m_{2g-1}\sinh(\log q^\omega))c_{g-1} \\
m_{2g-1}^{-1}c_g \\
m_{2g-1}^{-1}(\cosh(\log q^\omega)-\omega^{-1} m_{2g-1}\sinh(\log q^\omega))c_{g-1} \\
\vdots \\
m_{2g-1}^{-1}(\cosh((g-2)\log q^\omega)-\omega^{-1} m_{2g-1}\sinh((g-2)\log q^\omega))c_2 \\
m_{2g-1}^{-1}(\cosh((g-1)\log q^\omega)-\omega^{-1} m_{2g-1}\sinh((g-1)\log q^\omega))c_1 
\end{bmatrix}
\endaligned
\]
with
\[
m_{2g-1}=\omega \frac{q^{g\omega}+q^{-g\omega}}{q^{g\omega}-q^{-g\omega}} = \omega \coth(g \log q^\omega). 
\]
By using 
\[
\aligned
\lim_{x \to 0^+}\Bigl[ \cosh((g-k)x)+\coth(gx)\sinh((g-k)x) \Bigr] &= \frac{2g-k}{g} \quad (0 \leq k \leq g), \\
\lim_{x \to 0^+}\Bigl[ \cosh((g-k)x)-\coth(gx)\sinh((g-k)x) \Bigr] &= \frac{k}{g} \quad (1 \leq k \leq g-1), \\
\lim_{x \to 0^+} m_{2g-1}^{-1}=\lim_{x \to 0^+} \frac{\log q}{x} \tanh(gx) = g \log q
\endaligned 
\]
for $x=\log q^\omega$ 
we obtain
\[
\lim_{q^\omega \to 1^+}
\begin{bmatrix}
\tilde{{\mathbf a}}_{g,\omega}(1) \\
\tilde{{\mathbf b}}_{g,\omega}(1)
\end{bmatrix}
=
\begin{bmatrix}
{\mathbf a}_g(1) \\
{\mathbf b}_g(1)
\end{bmatrix}.
\]
On the other hand, it is easy to see that 
\[
\lim_{q^\omega \to 1^+} \tilde{m}_{2g}(\underline{c}\,;q^\omega) = \frac{1}{g \log q} = m_{2g}(\underline{c},\,\log q).
\]
Hence $\lim_{q^{\omega} \to 1^+} R_{n}(\underline{c}\,;q^{\omega})=R_{n}(\underline{c})$ for every $1 \leq n \leq 2g$. 
Because $R_{n}(\underline{c}\,;q^{\omega})$ are rational function of $q^\omega$, 
we obtain the second formula of Theorem \ref{thm_5}. Therefore, we obtain 
\[
\aligned
m_{2g-n}&(\underline{c}\,;q^\omega) =m_{2g}(\underline{c}\,;q^\omega)R_{n-1}(\underline{c}\,;q^\omega)R_n(\underline{c}\,;q^\omega) \\ 
& = \left( \frac{1}{\omega g \log q} + O(\log q^\omega) \right)
\Bigl( R_{n-1}(\underline{c}) +O(\log q^\omega) \Bigr) \Bigl( R_n(\underline{c}) + O(\log q^\omega) \Bigr).  
\endaligned
\]
This implies the third formula of Theorem \ref{thm_5}. \hfill $\Box$
%
%
\subsection{Remark on Theorem \ref{thm_5}} 
%
%
We have 
\[
E_{q,\omega}(z) = A_q(z) - i \omega B_q(z) + O_z(\omega^2), 
\]
\[
A_{q,\omega}(z) = A_q(z) + O_z(\omega^2), \quad 
B_{q,\omega}(z) = \omega B_q(z) +O_z(\omega^3)
\]
as $\omega \to 0^+$ if $z$ lie in a compact subset of $\C$. 
Therefore, it seems that $E_{q,\omega}(z)$ is similar to $E_q(z)=A_q(z) - i B_q(z)$ for small $\omega>0$, 
but there is an obvious gap after taking the limit $\omega \to 0^+$. 
To resolve this  gap, we consider 
\[
\tilde{E}_{q,\omega}(z) 
:= A_{q,\omega}(z) - \frac{i}{\omega} \, B_{q,\omega}(z). \\
\]
Then, we have
\[
\aligned
\tilde{E}_{q,\omega}(z) 
&= A_q(z) - i\,B_q(z) + O_z(\omega^2) 
= E_q(z) + O_z(\omega^2),
\endaligned
\]
\[
\tilde{A}_{q,\omega}(z) 
:= \frac{1}{2}(\tilde{E}_{q,\omega}(z)+\tilde{E}_{q,\omega}^\sharp(z)) 
= A_{q,\omega}(z) 
= A_q(z) + O_z(\omega^2),
\]
\[
\tilde{B}_{q,\omega}(z)
:= \frac{i}{2}(\tilde{E}_{q,\omega}(z)-\tilde{E}_{q,\omega}^\sharp(z))
 = \frac{1}{\omega}B_{q,\omega}(z) 
 = B_q(z) +O_z(\omega^2)
\]
as $\omega \to 0^+$ if $z$ lie in a compact subset in $\C$. 
Hence $\tilde{E}_{q,\omega}(z)$ ``recovers'' $E_q(z)$ by taking the limit $\omega \to 0^+$. 
The initial vector $\tilde{v}_g(0)$ of Section \ref{section_7_3} is chosen 
so that it corresponds to $\tilde{A}_{q,\omega}(z)$ and $\tilde{B}_{q,\omega}(z)$. 
This is a reason why Theorem \ref{thm_5} holds. 
In spite of this advantage, we chose $E_{q,\omega}(z)$ not $\tilde{E}_{q,\omega}(z)$ 
to state results in Section \ref{section_2_3}.  
One of the reason is the simple formula $E_{q,\omega}(z)=A_q(z+i\omega)$. 
Comparing this, $\tilde{E}_{q,\omega}(z)$ has a slight complicated formula 
\[
\tilde{E}_{q,\omega}(z) 
= \frac{A_{q}(z+i\omega)+A_{q}(z-i\omega)}{2} + \frac{A_{q}(z+i\omega)-A_{q}(z-i\omega)}{2\omega}.
\]
We do not know whether an analogue of Lemma \ref{lem_2_1} holds for $\tilde{E}_{q,\omega}(z)$. 
However, if such analogue holds, 
we may obtain results for $\tilde{E}_{q,\omega}(z)$ analogous to results in Section \ref{section_2_3}. 

%
%
\section{Concluding Remarks} \label{section_8}
%
%
By Corollary \ref{cor_602} and the proof of Section \ref{section_6_3}, 
a self-reciprocal polynomial $P_g(x)$ of \eqref{def_Pg} has 
only simple zeros on the unit circle if and only if 
there exists $2g$ positive real numbers $m_0,\cdots,m_{2g-1}$ such that 
\[
P_g(x) 
= \frac{P_g(1)}{2^{2g}} 
[\,1\,~\,0\,]
\prod_{n=1}^{2g}
\begin{bmatrix}
(x+1) & im_{2g-n}(x-1)\\
-im_{2g-n}^{-1}(x-1) & (x+1)
\end{bmatrix}
\begin{bmatrix}
1 \\ 0
\end{bmatrix},
\]
since $P_g(1)=E_q(0)=A_q(0)$. 
Compare this with the factorization  
\[
P_g(x) 
= P_g(0) \prod_{j=1}^{g}(x^2-2\lambda_j x+1) \quad (\lambda_j \in \R).
\]
As in Section \ref{section_2}, we have a simple algorithm to calculate $m_0,\cdots,m_{2g-1}$ from coefficients $c_0,\cdots,c_g$, 
but it is not for $\lambda_1,\cdots,\lambda_{g}$, since the Galois groups of a general self-reciprocal polynomial $P_g(x)$  
is isomorphic to ${\frak S}_g \ltimes (\Z/2\Z)^g$. 
In addition, it is understood that the positivity of $m_0,\cdots,m_{2g-1}$ is equivalent to the positivity of a Hamiltonian, 
but a plausible meaning of $|\lambda_j|<1$ and $\lambda_i \not= \lambda_j$ ($i \not= j$) is not clear. 
\medskip

Before concluding this paper, 
we remark two important remaining problems. 
\medskip 

(1) As mentioned in the end of Section \ref{section_2_2}, 
general closed formula of $m_{2g-n}(\underline{c},\;\log q)$ or $R_n(\underline{c})$  
is not yet obtained even for conjectural one. 
Therefore, the discovery of such formula is desirable. 
In fact, we need a simple closed formula of $m_{2g-n}(\underline{c},\;\log q)$ or $R_n(\underline{c})$ 
for the convenience of actual applications of Theorem \ref{thm_1} or \ref{thm_2}. 
The following formula of $R_n(\underline{c})$ 
may suggests the existence of certain ``nice'' general formula of $R_n(\underline{c})$, 
although an expected result is not yet clear. We have 
\[
P_g(x)= \sum_{k=0}^{g-1}c_k(x^{2g-k}+x^{k}) + c_g x^g
= c_0 \prod_{j=1}^{g}(x^2-2\lambda_j x+1), 
\]
where $2\lambda_j$ ($1 \leq j \leq g$) are 
zeros of the Chebyshev transform of $P_g(x)$ (\cite{Lakatos02}). 
By using $\lambda_1,\cdots,\lambda_g$, we obtain the following formula of $R_n(\underline{c})$: \\
$\bullet$ $g=1$, $\displaystyle{
R_{2}(c_0,c_1) = \frac{1 - \lambda_1}{1 + \lambda_1}}$, \\
$\bullet$ $g=2$, $R_{n}=R_{n}(c_0,c_1,c_2)$ ($2 \leq n \leq 4$), 
\[
R_{2} = \frac{(1 - \lambda_1)+(1 - \lambda_2)}{(1 + \lambda_1)+(1 + \lambda_2)}, \quad  
R_{3} = 2\frac{(1 - \lambda_1^2)+(1 - \lambda_2^2)}{(\lambda_1-\lambda_2)^2}, \quad 
R_{4} = \frac{(1 - \lambda_1)(1 - \lambda_2)}{(1 + \lambda_1)(1 + \lambda_2)},
\]
$\bullet$ $g=3$, $R_{n}=R_{n}(c_0,c_1,c_2,c_3)$ ($2 \leq n \leq 6$), 
\[
\aligned
R_{2} &= \frac{(1 - \lambda_1)+(1 - \lambda_2)+(1 - \lambda_3)}{(1 + \lambda_1)+(1 + \lambda_2)+(1 + \lambda_3)}, \quad 
R_{3} = 3 \frac{(1-\lambda_1^2)+(1-\lambda_2^2)+(1-\lambda_3^2)}{(\lambda_1-\lambda_2)^2+(\lambda_1-\lambda_3)^2+(\lambda_2-\lambda_3)^2}, \\
R_{4} &= \frac{\sum_{1 \leq i<j \leq 3}(1-\lambda_i)(1-\lambda_j)(\lambda_i-\lambda_j)^2}
{\sum_{1 \leq i<j \leq 3}(1+\lambda_i)(1+\lambda_j)(\lambda_i-\lambda_j)^2},  
~R_{5} = 3 \frac{\sum_{1 \leq i<j \leq 3}(1-\lambda_i^2)(1-\lambda_j^2)(\lambda_i-\lambda_j)^2}
{\prod_{1 \leq i<j \leq 3}(\lambda_i-\lambda_j)^2}, \\
R_{6} &= \frac{(1 - \lambda_1)(1 - \lambda_2)(1 - \lambda_3)}{(1 + \lambda_1)(1 + \lambda_2)(1 + \lambda_3)}.
\endaligned
\]
These formulas look simpler than formulas of $R_n(\underline{c})$ in the end of Section \ref{section_2_2}
\begin{remark} 
After completing the first version of this paper, 
simple closed formulas of 
$m_{2g-n}(\underline{c},\;\log q)$, $m_{2g-n}(\underline{c},\;q^\omega)$, 
$R_n(\underline{c})$ and $R_n(\underline{c}\,;q^\omega)$ are obtained. 
See \cite{Suzuki12} for details. 
\end{remark}
\medskip

(2) There are several important classes of 
self-reciprocal polynomials with real coefficients. 
Here we mention two of them. 
The first one is 
zeta functions of smooth projective curves $C/{\mathbb F}_q$ of genus $g$:
$Z_C(T)=Q_C(T)/((1-T)(1-qT))$,
where $Q_C(T)$ is a polynomial of degree $2g$ satisfying the functional equation $Q_C(T)=(q^{1/2} T)^{2g} Q_C(1/(qT))$. 
Hence $P_C(x)=Q_C(q^{-1/2}x)$ is a self-reciprocal polynomial of degree $2g$ with real coefficients. 
Weil~\cite{Weil45} proved that all zeros of $P_C(x)$ lie on $T$ 
as a consequence of Castelnuovo's positivity for divisor classes on $C \times C$. 
The second one is polynomials $P_A(x)$ 
attached to $n \times n$ real symmetric matrices $A=(a_{i,j})$ 
with $|a_{i,j}| \leq 1$ for every $1 \leq i<j \leq n$ (no condition on the diagonal):
$P_A(x) = \sum_{I \sqcup J =\{1,2,\cdots,n\}} x^{|I|} \prod_{i \in I, j \in J} a_{i,j}$, 
where $I \sqcup J$ means a disjoint union. 
Polynomials $P_A(x)$ are obtained as the partition function of a ferromagnetic Ising model
and they are self-reciprocal polynomials of degree $n$ with real coefficients. 
The fact that all zeros of any $P_A(x)$ lie on $T$ 
is known as the  Lee-Yang circle theorem~\cite{LeeYang52}. 
Ruelle~\cite{Ruelle71} extended this result and characterized polynomials $P_A(x)$ in terms of multi-affine polynomials 
being symmetric under certain involution on the space of multi-affine polynomials ~\cite{Ruelle10}. 

It seems that a discovery of arithmetical, geometrical or physical interpretation 
of the positivity of $m_{2g-n}(\underline{c},\,;\log q)$ or $R_n(\underline{c})$ 
(for some restricted class of polynomials) is quite interesting and important problem. 
Such philosophical interpretation  may contribute to find a simple closed formula of $m_{2g-n}(\underline{c},\,;\log q)$ or $R_n(\underline{c})$. 

%


\bigskip \noindent
Department of Mathematics,  
Tokyo Institute of Technology \\
2-12-1 Ookayama, Meguro-ku, 
Tokyo 152-8551, JAPAN  \\
Email: {\tt msuzuki@math.titech.ac.jp}


\begin{thebibliography}{99}
%
\bibitem{Chen95}
W. Chen,
\newblock{On the polynomials with all their zeros on the unit circle}, 
\newblock{\it J. Math. Anal. and Appl.} 
\newblock{{\bf 190} (1995), 714--724}.
%
\bibitem{Chinen08}
K. Chinen,
\newblock{An abundance of invariant polynomials satisfying the Riemann hypothesis}, 
\newblock{\it Discrete Math.}
\newblock{{\bf 308} (2008), no. 24, 6426--6440}.
%
\bibitem{Cohn22} 
A. Cohn, 
\newblock{{\"U}ber die Anzahl der Wurzeln einer algebraischen Gleichung in einem Kreise}, 
\newblock{\it Math. Z.} 
\newblock{{\bf 14} (1922), no. 1, 110--148}.
%
\bibitem{deBranges68} 
L. de Branges, 
\newblock{Hilbert spaces of entire functions}, 
\newblock{\it Prentice-Hall, Inc., Englewood Cliffs, N.J.}, 
\newblock{1968}.
%
\bibitem{Dym70} 
H. Dym, 
\newblock{An introduction to de Branges spaces of entire functions with applications to differential equations of the Sturm-Liouville type}, 
\newblock{\it Advances in Math.} 
\newblock{{\bf 5} (1970),395--471}.
%
\bibitem{Kwon11}
D. Y. Kwon, 
\newblock{Reciprocal polynomials with all zeros on the unit circle}, 
\newblock{\it Acta Math. Hungar.}
\newblock{{\bf 131} (2011), no. 3, 285--294}.
%
%
\bibitem{Lagarias06}
J. C. Lagarias, 
\newblock{Hilbert spaces of entire functions and Dirichlet $L$-functions}, 
\newblock{\it Frontiers in number theory, physics, and geometry. I}, 
\newblock{365--377}, 
\newblock{\it Springer, Berlin}, 
\newblock{2006}.
%
\bibitem{Lagarias09}
\bysame, 
\newblock{The Schr{\"o}dinger operator with Morse potential on the right half-line}, 
\newblock{\it Commun. Number Theory Phys.}, 
\newblock{{\bf 3} (2009), no. 2, 323--361}.
%
\bibitem{Lakatos02}
P. Lakatos, 
\newblock{On zeros of reciprocal polynomials}, 
\newblock{\it Publ. Math. Debrecen}
\newblock{{\bf 61} (2002), no. 3-4, 645--661}.
%
\bibitem{LakatosLosonczi09}
P. Lakatos, L. Losonczi, 
\newblock{Polynomials with all zeros on the unit circle}, 
\newblock{\it Acta Math. Hungar.}
\newblock{{\bf 125} (2009), no. 4, 341--356}.
%
\bibitem{LalinSmyth12}
M. N. Lal{\'i}n, C. J. Smyth,
\newblock{Unimodularity of zeros of self-inversive polynomials}, 
\newblock{\it Acta Math. Hungar.}
\newblock{(2012)},
\newblock{DOI:\texttt{10.1007/s10474-012-0225-4}}
%
\bibitem{LeeYang52}
T. D. Lee, C. N. Yang, 
\newblock{Statistical theory of equations of state and phase transitions. II. Lattice gas and Ising model}, 
\newblock{\it Physical Rev. (2)}
\newblock{{\bf 87} (1952), 410--419}.
%
\bibitem{Levin80}
B. Ja. Levin, 
\newblock{Distribution of zeros of entire functions}, 
\newblock{Revised edition, Translations of Mathematical Monographs, 5}, 
\newblock{\it American Mathematical Society, Providence, R.I.}, 
\newblock{1980}. 
%
\bibitem{Levin96}
\bysame, 
\newblock{Lectures on entire functions}, 
\newblock{Translations of Mathematical Monographs, 150}, 
\newblock{\it American Mathematical Society, Providence, RI}, 
\newblock{1996}. 
%
\bibitem{Lucas74}
F. Lucas, 
\newblock{Propri\'et\'es g\'eom\'etriques des fractions rationnelles}, 
\newblock{\it C. R. Acad. Sci. Paris} 
\newblock{{\bf 77} (1874), 431--433; {\bf 78} (1874), 140--144; {\bf 78} (1874), 180--183; {\bf 78} (1874), 271--274}.
%
\bibitem{Marden66}
M. Marden, 
\newblock{Geometry of polynomials, 2nd ed.}, 
\newblock{Mathematical Surveys 3}, 
\newblock{\it American Mathematical Society, Providence, R.I.},
\newblock{1966}.
%
\bibitem{MiloRass00}
G. V. Milovanovi\'c, T. M. Rassias, 
\newblock{Distribution of zeros and inequalities for zeros of algebraic polynomials}, 
\newblock{\it Functional equations and inequalities}, 
\newblock{171--204, Math. Appl., 518}, 
\newblock{\it Kluwer Acad. Publ., Dordrecht}, 
\newblock{2000}. 
%
\bibitem{Remling02}
C. Remling,
\newblock{Schr{\"o}dinger operators and de Branges spaces}, 
\newblock{\it J. Funct. Anal.}
\newblock{{\bf 196} (2002), no. 2, 323--394}.
%
\bibitem{Ruelle71}
D. Ruelle, 
\newblock{Extension of the Lee-Yang circle theorem}, 
\newblock{\it Phys. Rev. Lett.}
\newblock{{\bf 26} (1971), 303--304}.
%
\bibitem{Ruelle10}
\bysame,
\newblock{Characterization of Lee-Yang polynomials}, 
\newblock{\it Ann. of Math. (2)}
\newblock{{\bf 171} (2010), no. 1, 589--603}.
%
\bibitem{Schur17}
I. Schur,
\newblock{\"Uber Potenzreihen, die im Innern des Einheitskreises beschr\"ankt sind}, 
\newblock{\it J. Reine Angew. Math.}
\newblock{{\bf 147} (1917), 205--232}.
%
\bibitem{Suzuki12}
M. Suzuki,
\newblock{On zeros of self-reciprocal polynomials. II}, 
\newblock{in preparation}, \\
\newblock{a draft is available at \url{http://www.math.titech.ac.jp/~msuzuki/srp2.pdf}}
%
%
\bibitem{Weil45}
A. Weil, 
\newblock{Sur les courbes alg\'ebriques et les vari\'et\'es qui s'en d\'eduisent}, 
\newblock{Actualit\'es Sci. Ind., no. 1041= Publ. Inst. Math. Univ. Strasbourg 7 (1945)}, 
\newblock{\it Hermann et Cie., Paris}, 
\newblock{1948}.
%
\end{thebibliography}
\end{document}